\documentclass[12pt, leqno]{article}
\usepackage{hyperref}
\usepackage{fullpage}
\usepackage{palatino}
\usepackage{amsmath,amssymb,latexsym,amsthm}
\usepackage{amsmath,amssymb}
\usepackage{graphicx}
\usepackage{color, epsfig}
\usepackage{subfigure}
\usepackage{mathrsfs}
\usepackage{cleveref}
\numberwithin{equation}{section} %
\newcommand{\Epi}{\operatorname{Epi}}
\newcommand{\Hypo}{\operatorname{Hypo}}
\newcommand{\Lip}{\operatorname{Lip}}
\providecommand{\R}{\mathbb{R}}

\providecommand{\N}{\mathbb{N}}

\DeclareMathOperator*{\OOO}{\overline{\Omega}}

\DeclareMathOperator*{\eps}{\varepsilon}

\def\la{ \lambda }

 \newcommand{\be}{\begin{equation}}
 \newcommand{\ee}{\end{equation}}
\def\for{\hskip0.9pt|\hskip0.9pt}

 \def\al{\alpha}
\def\traj{X^x_\al}
\def\val{V_\ld}
\def\ld{\lambda}
\def\vall{\ld\val}
\def\vl{v_\la}
\def\wl{w_\ld}

\def\rn{{\mathbb{R}}^n}
\def\bO{\overline{\Omega}}

\newtheorem{theorem}{Theorem}[section]
\newtheorem{proposition}[theorem]{Proposition}
\newtheorem{lemma}[theorem]{Lemma}
\newtheorem{corollary}[theorem]{Corollary}

\theoremstyle{definition}
\newtheorem{definition}[theorem]{Definition}
\theoremstyle{remark}
\newtheorem{remark}[theorem]{Remark}
\newtheorem{remarks}[theorem]{Remarks}
\newtheorem{example}[theorem]{Example}
\newcommand{\beq}{\begin{equation}}

\renewcommand{\geq}{\geqslant}
\renewcommand{\leq}{\leqslant}

\newcommand{\cont}{\mathcal{C}}
\newcommand{\bounded}{\mathcal{B}}

\newcommand{\olim}{\operatorname{o-lim}}
\def\<#1,#2>{\langle#1,#2\rangle}
\providecommand{\R}{\mathbb{R}}

\providecommand{\N}{\mathbb{N}}

\newcommand{\discount}{\alpha}
\theoremstyle{definition}
\theoremstyle{remark}
\usepackage[colorinlistoftodos,bordercolor=orange,backgroundcolor=orange!20,linecolor=orange,textsize=scriptsize]{todonotes}

\title{\bf Analysis of the vanishing discount limit \\for optimal control problems\\ in continuous  and discrete time}

\author{ P. Cannarsa \thanks{Dipartimento di Matematica, Universit\`a di Roma 'Tor Vergata', Via della Ricerca Scientifica 1, I-00133 Roma, Italy (cannarsa@mat.uniroma2.it)}\and  S. Gaubert\thanks{INRIA and CMAP, \'Ecole polytechnique, IP Paris, CNRS, UMR 7641, Route de Saclay, 91128 Palaiseau Cedex, France}\and
C. Mendico \thanks{ Dipartimento di Matematica, Universit\`a di Roma 'Tor Vergata', Via della Ricerca Scientifica 1, I-00133 Roma, Italy (mendico@mat.uniroma2.it)} \and
M.
Quincampoix\thanks{Univ.\ Brest, CNRS UMR 6205, Laboratoire de Mathematiques de Bretagne Atlantique     6, Avenue Victor    Le Gorgeu,     29200 Brest,     France (marc.quincampoix@univ-brest.fr) }
}

\begin{document}

\maketitle

\begin{abstract}
A classical problem in ergodic continuous time control consists of studying the limit behavior of the optimal value of a discounted cost functional with infinite horizon  as the discount factor $\lambda$  tends to zero. In the literature, this problem has been addressed under various controllability or ergodicity conditions ensuring that the rescaled value function converges uniformly to a constant limit. In this case the limit can be characterized as the unique constant such that a suitable Hamilton-Jacobi equation has at least one continuous viscosity solution.
In this paper, we study this problem without such conditions, so that the aforementioned limit needs not be constant. Our main result characterizes
the uniform limit (when it exists)  as the maximal subsolution of a system of Hamilton-Jacobi equations. Moreover, when  such a subsolution is a viscosity solution, we obtain the convergence of optimal values as well as a rate of convergence. This mirrors
the analysis of the discrete time case, where
we characterize the uniform limit as the supremum over a set
of sub-invariant half-lines of the dynamic programming operator.
The emerging structure in both discrete and continuous time models shows that the supremum over   sub-invariant half-lines with respect to the Lax-Oleinik semigroup/dynamic programming operator, captures the behavior of the limit cost as discount vanishes.

\end{abstract}



\section{Introduction}

For any given $\lambda>0$ let
$\val $ denote the  value function of
 the  optimal  control problem
\begin{equation}
\label{eq:val}
\val (x) := \inf_{\al \in {\cal A}} \;   \int_{0}^ \infty e ^{- \lambda t}  L\big(\traj (t), \al (t)\big)dt
\qquad (x\in \Omega \subset  \rn)
\end{equation}
where ${\cal A}$ stands for the set of all measurable controls $\al :[0,\infty)\to A$ taking values in a complete metric space  $A$ and  $ \traj(\cdot)  $ for the solution of the state equation
\begin{equation}\label{eq:traj}
\begin{cases}
\hspace{.cm}
X'(t)=f\big(X(t), \al (t)\big),& t>0
\\
X(0)=x \enspace .\hspace{.cm}
\end{cases}
\end{equation}
The assumptions we give below ensure that (\ref{eq:traj}) has a unique solution and that  $\val $ satisfies the following Hamilton-Jacobi equation \be \label{EDP}
\lambda V (x) + H \big(x, -\nabla V (x) \big) = 0\qquad(x \in \overline{\Omega})
\ee where the {\em Hamiltonian} is given by
\begin{equation}
\label{H}
H(x, p) :=\max_{a \in A} \big\{ \big\langle f(x,a),p \big\rangle  -L(x,a) \big\}\qquad\forall\, (x,p) \in  \rn \times  \rn.
\end{equation}
We are interested in studying the set of cluster points ---in the  topology of uniform conver\-gence--- of the family $\{\vall\}_{\ld>0}$ as $ \lambda $ tends to $0 ^+$.

  Some of the main ideas of this work will become transparent when considering in parallel an analogous control problem in discrete time.
 In such a model, we suppose that the space state $S$ is a compact (Hausdorff) topological space, and that the evolution of the value function is described by a dynamic programming operator $T$, acting on the Banach space $X=\cont(S)$ of real values continuous functions on $S$, equipped with the sup-norm. In the case of a single player, the operator $T$ is a sup-norm nonexpansive mapping of the form
  as
  $$
T_i(x) = \inf_{a\in A_i} \big(r^a_i  + P^a_i x),\qquad i\in S $$
where for any state $i \in S$, $T_i  : X \to \R$ is the $i$-coordinate map of $T$,  $A_i$ is a non-empty set of actions, $r^a_i \in \R$ is a payment
depending on the state and action, and $P^a_i$ a Markov kernel,
i.e., a positive continuous linear map $X\to \R$ such that $P^a_i e =1$,
where $e$ denotes the unit function (identically one). The operator $T$ may be thought of as a discrete time analogue of the Lax-Oleinik semigroup,
i.e., of the evolution semigroup of the PDE
\[
\frac{\partial v}{\partial t} + H(x,- \nabla v(x)) = 0  \enspace .
\]

The discrete-time analogue of the discounted value
function $V_ \lambda$  which satisfies  (\ref{EDP}),  is the vector $v_\discount\in X$ defined, for a discount factor
$1-\discount$ with $0<\discount<1$,  by
\[
T((1-\discount) v_\discount) = v_\discount
\]
The correspondence between the discrete and continuous time case
is obtained by setting $\exp(-\lambda)=1-\discount$.
We are interest ed in studying the set of cluster points---in the  topology of uniform convergence--- of the family $  \discount v_\discount $ as $ \discount$ tends to $0 ^+$.

In the continuous case,
we will show that  the set of cluster points---in the  topology of uniform convergence -- of the family $\{\vall\}_{\ld>0}$ consists of at most a singleton, which we characterize in terms of subsolutions of a system of Hamilton-Jacobi equations.
This provides an alternative to the characterization
in terms of invariant measures addressed in \cite{BQR}.  Various conditions that guarantee that the set of cluster points is nonempty can be found in~\cite{CQ, QR}.

The analysis developed in this paper is a counterpart of weak KAM theory and Aubry-Mather theory, see for instance \cite{MS} and  \cite{DFIM}, in which the long-time average behavior, and respectively the vanishing discount problem for Hamilton-Jacobi equations, is studied for calculus of variation problems. We recall that in this setting, it has been proved,
under a controllability assumption (coercivity of the Hamiltonian in the adjoint
variable),
that the sequence $\{ \lambda V_{\lambda}\}_{\lambda >0}$ locally uniformly converges
to a constant function, the so-called Ma\~n\'e critical value,
which can be represented  as the minimal action over the set of invariant probability measures   under the Lagrangian flow, see \cite{Fathi}.
Moreover,  the sequence $\{ V_{\lambda}\}_{\lambda >0}$, up to additive constants,
locally uniformly converges to a solution of the so-called critical Hamilton-Jacobi equation. We also refer to \cite{A1, A2, barles} for a PDE approach and to \cite{BV, Col, GS,  Grune, Wirth, bib:Cannarsa_Mendico} for a study based on properties of trajectories of control systems. 
 One of the main results of this paper is a quite general characterization of the set of cluster points of the sequence $\{ \lambda V_{\lambda}\}_{\lambda > 0}$ (\Cref{representationI}), that we show to be at most a singleton. In particular, we prove that such a single function is given by
the supremum of all $v \in \mathcal C ( \overline{\Omega})$ for which one can find another function $u \in \mathcal C ( \overline{\Omega})$  such that the pair $(u,v)$ satisfies the following system of Hamilton Jacobi inequalities  in the viscosity sense 
 \begin{equation*}
\begin{cases}
 h(x, -\nabla v (x)) \leq 0
 \\
 v+H(x,-\nabla u (x)) \leq 0
\end{cases}
\quad\text{on}\;\;\overline { \Omega },
\end{equation*}  where $ 
h(x, p): =\max_{a \in A}\, \langle f(x,a),p \rangle $. 
 In our opinion, one interesting feature of the above result, compared with classical weak KAM theory,  is the fact that the limit of $\lambda V_{\lambda}$---when it exists---is not a constant but a function.   In the ergodic case (namely when $v$ is a constant $c$), the second line of the above system of Hamilton Jacobi  inequalities  is related  to  the well-known  equation $ c+ H(x,-\nabla u (x)) =0$ studied for ergodic control  \cite{AB} or homogeneization \cite{LPV}. In this case, where $v$ is a constant, the first line of the system does not appear because it is trivially satisfied. Moreover, the solvability
of a variant of the above system, with equalities instead of inequalities, yields an estimate for the rate of convergence of $\lambda V_{\lambda}$ to its limit as $\lambda\downarrow 0$, see Theorem \ref{speed}.
Furthermore, for specific control systems, we prove that the limit  of $\lambda V_{\lambda}$  could be expressed  as the infimum over the set of reachable points of the minimal Lagrangian, see Proposition~\ref{representationII}. However, the existence of solutions to the above system is strictly related to the existence of the (possibly) uniform limit of $V_{\lambda}$ as $\lambda \downarrow 0$. This is known to be a challenging and still an open problem when the state equation is nonlinear in space and in control due to, for instance, the lack of small time controllability of the underlying controlled dynamical system.

We refer the reader to \cite{BQR, BLQR} for an alternative characterization of the limit function in terms of invariant measures for the controlled dynamics.  We point out that, in order to guarantee the existence of invariant measures, the authors suppose that there exists an invariant set $\OOO \subset \R^{d}$ for the dynamics, i.e., for any initial point in $\OOO$ and any control function the solution of the state equation remains in $\OOO$. We also refer the reader to \cite{CQ, QR} for various conditions to guarantee that the set of cluster points of $\{ \lambda V_{\lambda} \}_{\lambda > 0}$ in nonempty.

We also provide an analogue of our main result in the discrete time setting
(see~\Cref{th-2}). Here, the key notion is that of sub-invariant half-line
of the operator $T$. Such a half-line refers to a function of a real
positive parameter $s$, of the form $s\mapsto u+ s\eta$ where $u,\eta\in X$,
and such that
\[ T(u+s\eta)\geq u+(s+1)\eta, \qquad \forall s\geq 0\enspace,
\]
so that $u$ can be thought of
as a ``basepoint'' and $\eta$ as an ``offset'' or ``director vector''.
We show that the the set of cluster points---in the  topology of uniform convergence -- of the family $  \discount v_\discount $ consists of at most a single point, which coincides with the supremum of the offsets $\eta$, among the
set of all sub-invariant half-lines. Subsolutions of the PDE system
are the continuous-time counterparts of sub-invariant half-lines.

The notion of sub-invariant half-line is inspired by a result of Kohlberg~\cite{kohlberg}, showing that, in finite dimension, polyhedral nonexpansive operators $T$
admit invariant half-lines (satisfying the above condition with equality),
and that the offset characterizes the limit. However, the condition
of existence of invariant half-lines (the polyhedrality of the operator $T$)
is extremely demanding. In contrast, the present relaxed
approach using sub-invariant half-lines works in full generality for one
player problems. Another source of inspiration, still in the finite
dimensional case, is the theory of the ``lexicographic system'' developed
in particular in~\cite{denardofox} and in~\cite{dynkin} to study the ergodic problem for Markov decision processes,
still in the polyhedral
(finite state and action spaces) setting. The present PDE system that we use
provides an infinitesimal version of the lexicographic system.

\medskip
The paper is organized as follows. In \Cref{sec-prelim}, we fix the setting of the first part of the paper, that is, notations, assumptions and we recall some preliminaries on Hamilton-Jacobi equations. In \Cref{sec-continuous}, we study the problem of characterizing the cluster point of $\{\lambda V_{\lambda}\}_{\lambda > 0}$, if it exists (see Theorem \ref{representationI}). In Section 4, we study the relation between such a cluster point  and certain systems of Hamilton-Jacobi equations $(S)$. Moreover, we give an estimate for the rate of convergence of $\lambda V_{\lambda}$, see Theorem \ref{speed}. In Section 5, we provide examples  of application of the analysis of the previous sections.
  Section 6 is devoted to the discrete time problem. We use a  general abstract  setting (AM-spaces with units) which encompasses, in particular, Bellman or Shapley operators over a compact state space (and, as a special case, Markov decision processes).  Section 7 contains our  main result for the discrete time case (Theorem  \ref{th-2}), which gives a characterization of the cluster points of $  \{\discount v_\discount \} _{\discount >0}$.
The appendix contains technical results needed in the paper, establishing, in particular, monotonicity properties for Hamilton Jacobi equations.

\section{Preliminaries}\label{sec-prelim}
We denote by $| x| $ the euclidean norm of $x \in  \R ^n$ and by $\<\cdot, \cdot> $ the associated scalar product. %

\subsection{Assumptions}
Let $f:  \rn\times U \to  \rn $ be a continuous map such that
\begin{equation}\label {eq:f}
\begin{cases}
(i)\hspace{.cm}
&
|f(x,a)|\leqslant  C
\\
(ii)\hspace{.cm}
&|f(x,a)-f(y,a)|\leqslant  C|x-y|
\end{cases}
\qquad\forall x,y\in \rn\,,\;\forall a\in A
\end{equation}
for some constant $C \geqslant 0$. Let   $L:  \rn \times A\to \R$ be a bounded continuous function satisfying
\begin{equation}\label{eq:L}
\begin{cases}
(i)\hspace{.cm}
&
0 \leqslant  L(x,a)  \leqslant  1
\\
(ii)\hspace{.cm}
&|L(x,a)-L(y,a)|\leqslant  k|x-y|
\end{cases}
\qquad\forall x,y\in \rn\,,\;\forall a\in A
\end{equation} for some constant $k \geqslant 0$.
We suppose throughout the article that
\begin{equation}\label {eq:conv}
\{ (f(x,a), L (x,a) +r ), \, a \in A, \, r \geq 0 \,\} \mbox{ is a convex set.}
\end{equation}

In this paper we assume the existence of a bounded open domain $\Omega \subset  \rn $ such that $\overline{\Omega}$ is {\em invariant} for system \eqref{eq:traj}, that is,
\begin{equation*}
x\in\overline{\Omega}\quad     \Longrightarrow\quad     \traj(t) \in \overline{\Omega}\,,\;\;\forall\, \al\in\mathcal A\,,\;\forall\, t\geqslant 0.
\eqno{(H_ \Omega )}
\end{equation*}

\subsection{Hamilton-Jacobi equation}
Owing to \eqref{eq:L} one can easily show that $ 0 \leqslant  \vall (x) \leqslant  1$ for all $x \in  \overline{\Omega }$. Moreover, under assumption $(H_ \Omega )$ it is also known that, for any fixed $\lambda >0$, the value function $\val$ defined in \eqref{eq:val} is continuous on $\overline{\Omega }$.
Associated with the minimization problem in \eqref{eq:val}, we consider the Hamilton-Jacobi equation (\ref{EDP}), whose solutions have to be understood as follows.

\begin{definition}\label{visc}
A continuous function $ V  : \overline{\Omega } \to \R$ ($V\in\mathcal C(\bO)$ in short) is a {\em viscosity solution} of \eqref{EDP} on $\overline{\Omega }$ if and only if:
\begin{itemize}
\item[$(a)$] $V$ is a viscosity {\em supersolution} on $\overline{\Omega }$, that is, $ \forall x \in \overline{\Omega }$
$$\lambda V(x) + H(x, -p) \geq 0 \qquad  \forall p \in \partial ^- _\Omega V (x),$$ where
\begin{equation*}
 \partial ^- _\Omega V (x) = \Big\{\; p \in \R ^n ~:~\liminf _{\bO\ni y \to x} \,\frac{V(y)-V(x) -\langle p, y-x\rangle }{|y -x|}\, \geq 0 \;\Big\}
\end{equation*} is the {\em subdifferential} of $V$ at $x$;
\item[$(b)$] $V$ is a viscosity {\em subsolution} on $\overline{\Omega }$, that is, $ \forall x \in \overline{\Omega }$
$$\lambda V(x) + H(x, -p) \leq 0 \qquad  \forall p \in \partial ^+ _\Omega V (x),$$ where
\begin{equation*}
 \partial ^+ _\Omega V (x) = \Big\{\; p \in \R ^n ~:~\limsup _{\bO\ni y \to x} \,\frac{V(y)-V(x) -\langle p, y-x\rangle }{|y -x|} \,\leq 0 \;\Big\}
\end{equation*}
is the {\em superdifferential} of $V$ at $x$.
\end{itemize}
\end{definition}
In the appendix
at the end of the paper, under assumption ($H _ \Omega $), we  will show that $ \val$ is the unique viscosity solution on $\overline{\Omega }$ of the Hamilton Jacobi equation \eqref{EDP}.
 \begin{remarks}
 \begin{enumerate}
 \item
Observe that when $x \in \Omega$, the sets $\partial _\Omega^+ V (x)$ and $\partial_\Omega^-V (x)$ reduce to the classical super and subdifferentials.

 \item The properties required by the above definition of solution are slightly different from the ones classically used for Hamilton-Jacobi equations with state constraints (cf \cite{BCD}), that is, )$V$ is a constrained viscosity solution if and only if it is a supersolution on $\overline{\Omega }$ and a subsolution on $\Omega $ (cf \cite{BCD}).

\item When  $\overline{\Omega }$ is invariant, a uniqueness result for constrained viscosity solutions of \eqref{EDP} can be obtained \cite{BCD} imposing regularity assumptions (interior cone condition) on the boundary of $ \Omega$. In the present paper, such regularity assumptions  are not needed.
 \end{enumerate}
\end{remarks}

\section{The vanishing discount limit in continuous time}\label{sec-continuous}
In this section, we analyse the behaviour of the rescaled family of value functions $\{\vall\}_{\ld>0}$ as the so-called discount rate $\ld$ goes to zero.
The existence of  cluster points of such a family, in the uniform topology on $\bO$, is usually deduced by compactness arguments showing that $\{\vall\}_{\ld>0}$ is equicontinuous on $\bO$. Equicontinuity can in turn be obtained by imposing further assumptions on the data, such as the coercivity of $H(x,\cdot)$ as in \cite{BCD} (which implies that the uniform limit is constant), or various types of nonexpansiveness conditions  (see \cite{QR, CQ} and the references therein). We do not address such an issue in this paper but we rather look for general
properties that can be deduced assuming the existence of cluster points or the equicontinuity of $\{\vall\}_{\ld>0}$.

\subsection{Cluster points of $\{\vall\}_{\ld>0}$}
We begin by analysing the properties of  cluster points  of the family $\{\vall\}_{\ld>0} \subset \mathcal C ( \overline{\Omega})$.
We  define the  {\em reduced Hamiltonian} by
\begin{equation}
\label{red_H}
h(x, p) =\max_{a \in A}\, \langle f(x,a),p \rangle \qquad \forall  (x,p) \in  \rn \times  \rn.
\end{equation}

\begin{proposition} \label{La}
If  $v^* $ is a cluster point of the family $\{\vall\}_{\ld>0}$ as $ \la \to 0 ^+ $ in the uniform topology  on  $\overline{\Omega}$, then $v^*$ is a viscosity solution of the equation
\begin{equation} \label{h}
h(x, -\nabla v^* (x)) =0 \quad \mbox{in}\quad \overline{\Omega}.
\end{equation}
\end{proposition}
\begin{proof}
  Let $\ld_j  \to 0 ^+ $ be such that   $ \ld_jV_{\la _j}  $ converges uniformly to $ v^*$.
By \eqref{EDP}  we have that
\begin{eqnarray*}
 0 &= &\ld_j  ^2 V_{\ld_j}+\ld_j H(x, -\nabla V_{\ld_j} )
 \\
& = &\ld_j^{2} V_{\ld_j}+ \max_{a \in A} \big\{ \big\langle f(x,a) , -\nabla (\ld_jV_{\ld_j})\big\rangle - \ld_j L(x,a) \big\}.
\end{eqnarray*}
So, it suffices to invoke the stability of viscosity solutions as $\lambda_{j} \to 0^{+}$ to conclude that
 \begin{equation*}
0 = \max_{a \in A} \big\langle f(x,a) ,   -\nabla v^*  (x) \big\rangle  = h(x, -\nabla  v^*  (x)) \quad \mbox{in}\quad \overline{\Omega}.
\end{equation*}
\end{proof}

We observe that the above property is usually proved in the open set $\Omega$ or assuming Dirichlet boundary conditions on $\partial \Omega$. However, as one can easily check, the proof  which is given in \cite{BCD} adapts unaltered to solutions on $\bO$---the case of interest to us.

The following system of Hamilton-Jacobi equations,  or, rather, inequalities, plays a major role in our approach.
\begin{definition}
We define ${\cal S}(\bO)$ to be the set of all pairs $(u,v) \in \mathcal C ( \overline{\Omega})\times  \mathcal C ( \overline{\Omega})$ which satisfy
 \begin{equation*}
\begin{cases}
 h(x, -\nabla v (x)) \leq 0
 \\
 v+H(x,-\nabla u(x) ) \leq 0
\end{cases}
\quad\text{on}\;\;\overline { \Omega }
\eqno{(H)}
\end{equation*}
in the viscosity sense, that is, in the sense of Definition~\ref{visc}~$(b)$.
\end{definition}

\begin{proposition}\label{pr:cluster_gives_subsolution}
 Let $\ld_j  \stackrel{j\to\infty}{\longrightarrow} 0 ^+ $ be such that   $ \ld_jV_{\la _j}  \stackrel{j\to\infty}{\longrightarrow} v^*$ uniformly on $\bO$. Then
 \begin{equation}\label{eq:cluster_gives_subsolution}
\big(V_{\ld_j},v^*-\epsilon_j\big)\in {\cal S}(\bO)\qquad\forall j\in\N
\end{equation}
where $ \epsilon _j = \| \ld_jV_ {\ld_j} - v^* \| _ \infty $.
\end{proposition}
\begin{proof}
Since $\val$ is the viscosity solution of  \eqref{EDP},
we have that
\begin{equation*}
H( x , -\nabla V_{\ld_j}) =  -\ld_jV_{\ld_j}  \leqslant -v^*  + \epsilon _j\quad\text{on}\;\;\overline { \Omega }
\end{equation*}
in the viscosity sense. Hence, Proposition~\ref{La} ensures that  $ v: = v^* - \epsilon _j $ and $ u := V_ {\lambda_j } $ satisfy
 $$h(x,- \nabla v ) \leqslant 0  \quad \mbox{and}  \quad v+H(x, -\nabla u )
   \leqslant 0\quad \text{on}\;\; \overline{\Omega}. 
$$
\end{proof}

\subsection{Characterization of the vanishing discount limit}
We now study  the subsolutions of $(H)$  which belong to
$$\Lip (\bO):=\Big\{ f\in \mathcal C(\bO)~:~\sup_{x\neq y}\,\frac{|f(y)-f(x)}{|y-x|}\,<\infty\Big\}.$$
\begin{proposition}\label{speed1}
 Let
 $$
 (\bar u,\bar v)\in \mathcal S_{\Lip}(\bO):=\big\{(u,v)\in  \mathcal S(\bO)~:~u,v\in \Lip(\bO)\big\}.
 $$
Then the following holds true for any $\ld>0$:
\begin{itemize}
\item[$(a)$] for any $c\in\R$ the function
\begin{equation}
\label{eq:def_w_c}
u_c(x):=\,\frac{\bar v(x)}\ld\,+\bar u(x)+c\quad (x\in\bO)
\end{equation}
satisfies
\begin{equation}
\label{eq:w_c}
\ld u_c+H(x,-\nabla u_c)\leq \ld(\bar u+c)  \quad\text{on}\;\;\overline { \Omega }
\end{equation}
in the viscosity sense;
\item[$(b)$] $\val$ satisfies the lower bound
\begin{equation}
\label{eq:lowerb}
\bar u(x)-\|\bar u\|_\infty\leq \val(x)-\,\frac{\bar v(x)}\ld\qquad\forall\,x\in\bO
\end{equation}
\end{itemize}
\end{proposition}
\begin{proof} Since $\bar u,\bar v\in \mathcal S_{\Lip}(\bO)$, Proposition~\ref{mono}~$(ii)$  below ensures that
for all $x \in  \overline{\Omega }$ and $ \alpha \in {\cal A}$
\begin{equation*}
\bar u (x)  \leq  \bar u  \big(\traj (t)\big) + \int _0 ^t  \big[L\big(\traj (s), \al (s)\big)-\bar v\big(\traj (s)\big)\big]ds
\qquad  \forall t \geq 0
\end{equation*}
and
\begin{equation*}
\bar v (x)  \leq  \bar v  (\traj (t) )
\qquad  \forall t \geq 0\,.
\end{equation*}
Hence,
\begin{equation*}
u_c (x)  \leq  u_c \big(\traj (t)\big) + \int _0 ^t  \big[L\big(\traj (s), \al (s)\big)-\bar v\big(\traj (s)\big)\big]ds
\qquad  \forall t \geq 0
\end{equation*}
which, again by Proposition~\ref{mono}~$(ii)$,  yields
$
 \bar v+H(x,-\nabla u_c) \leqslant  0\; \text{on}\; \overline{\Omega}
$
and  \eqref{eq:w_c}.

The proof is completed noting that $(b)$ is an easy consequence of $(a)$. Indeed, choosing $ c=-\|u\|_\infty$ in \eqref{eq:def_w_c} we deduce from \eqref{eq:w_c} that $u_{ c}$ is a subsolution of \eqref{EDP}. Therefore, $u_{ c}\leq \val$ on $\bO$ by comparison.\end{proof}

\begin{remark}\label{rm:Lipschitz_subsolution}
 In general, the existence of a cluster point $v^*$ for the family $\{\vall\}_{\ld>0}$ does not suffice to guarantee that
 $ \mathcal S_{\Lip}(\bO)\neq\varnothing$. However, if $v^*\in \Lip(\bO)$ then this is indeed the case because the first component, $V_{\ld_j}$, in \eqref{eq:cluster_gives_subsolution} is always Lipschitz.
\end{remark}

Our next result (Theorem \ref{representationI})
characterizes the unique, possible, cluster point of the family $\{\lambda V_{\lambda}\}_{\lambda > 0}$.

\begin{theorem}[{\bf Characterization}]\label{representationI}
Let  $v^*$ be a cluster point of $\{\vall\}_{\ld>0}$ as $ \la \downarrow 0$,  in the uniform topology on  $\overline{\Omega}$, such that $v^*\in \Lip(\bO)$. Then, we have that
\begin{equation}\label{R1}
v^{*}(x) = \sup \big\{\,v(x)~:~(u,v)\in {\cal S}_{\Lip} (\bO) \,\big\} \qquad  \big(x \in  \overline{\Omega}\big)
 \end{equation}
 \end{theorem}
\begin{proof} Define the function $v_{0}$ as the right-hand side of \eqref{R1}, that is,
\begin{equation*}
	v_{0}(x)= \sup \big\{\,v(x)~:~(u,v)\in {\cal S}_{\Lip} (\bO) \,\big\} \qquad  \big(x \in  \overline{\Omega}\big).
\end{equation*}
Let $\ld_j  \downarrow 0$ be such that   $\ld_jV_{\ld_j} $ converges uniformly to $  v^* $ as $j\to\infty$. First we prove that $v^*\leq v_0$.
In view of Proposition~\ref{pr:cluster_gives_subsolution}
and Remark~\ref{rm:Lipschitz_subsolution},
$(V_{\ld_j},v^*-\epsilon_j)$, with $ \epsilon _j = \| \ld_jV_ {\ld_j} - v^* \| _ \infty $, belongs to ${\cal S}_{\Lip}(\bO)$  for all $j\in\N$. Therefore $v^*-\epsilon_j\leq v_0$ for all $j$, which in turn yields $v^*\leq v_0$.

In order to prove the converse inequality, let $(\bar u,\bar v)\in {\cal S}_{\Lip} (\bO)$. Then, by \eqref{eq:lowerb} we obtain
\begin{equation*}
\bar v(x)\leq \ld_jV_{\ld_j}(x)+\ld_j(\|\bar u\|_\infty-\bar u(x))\qquad \forall x \in  \overline{\Omega}\,,\;\forall j\in \N\,.
\end{equation*}
Taking the limit as $j\to\infty$ it follows that $\bar v\leq v^*$ and this yields $v_0\leq v^*$.
\end{proof}

By Ascoli's Theorem and the above result we deduce the following.
\begin{corollary} Suppose  $\{\vall\}_{\ld>0}$  is equi-Lipschitz. Then $\vall$ converges uniformly as $ \la \downarrow 0$ to the function $v^{*}$ in \eqref{R1}.
\end{corollary}

Denote by $\mathcal{R}(x)$ the set of reachable points from $x \in \overline\Omega$, that is,
\begin{equation*}
\mathcal{R}(x)=\big\{X_{\alpha}^{x}(t): \alpha \in A,\ t \geq 0 \big\},
\end{equation*}
and define
\begin{equation*}
I(x):=\inf_{ y \in \mathcal{R}(x)} \min_{a \in A} L(y, a), \quad x \in \overline\Omega.
\end{equation*}

\begin{proposition}\label{representationII}
Assume $\{ \lambda V_{\lambda}\}_{\lambda >0}$ converges to  the function $v^{*}$  defined in Theorem \ref{representationI}.
Moreover, suppose that
\begin{itemize}
\item[{\bf (R)}] for any $x \in \overline\Omega$ and any $\eps > 0$ there exist $\alpha_{\eps} \in A$, $T_{\eps} > 0$ and $a_{\eps} \in A$ such that
\begin{equation*}
	 L(X_{\alpha_{\eps}}^{x}(T_{\eps}),a_{\eps})	 \leq I(x)+\eps \quad\mbox{and}\quad f(X_{\alpha_{\eps}}^{x}(T_{\eps}),a_{\eps})=0.
\end{equation*}
\end{itemize}
Then, $v^{*}(x)=I(x)$ for all $x \in \OOO$.
\end{proposition}
\begin{proof}
We will prove that $\lambda V_{\lambda}(x) \to I(x)$ as $\lambda \to 0$ which implies that $v^\star(x)=I(x)$ for any $x \in \OOO$.
Recall that $\lambda V_{\lambda}$ can be represented as follows
\begin{equation*}
\lambda V_{\lambda}(x) = \lambda \inf_{\alpha \in A}\left\{ \int_{0}^{\infty}{e^{-\lambda s}L(X_{\alpha}^{x}(s), \alpha(s))\ ds}\right\},
\end{equation*}
see for instance \cite[Proposition ~III 2.8]{BCD}.
Hence,  by the above representation formula and the definition of $I(x)$, we deduce that
\begin{equation*}
\lambda \int_{0}^{\infty}{e^{-\lambda s}L(X_{\alpha}^{x}(s), \alpha(s))\ ds} \geq \lambda \int_{0}^{\infty}{e^{-\lambda s}I(x)\ ds} = I(x),
\end{equation*}
or
\begin{equation*}
\lambda V_{\lambda}(x) \geq I(x), \quad x \in \OOO.
\end{equation*}

On the other hand,  for any $ \eps>0$ take  $T _ {\eps} \geq 0$  and $a_{\eps }\in A$ as in  {\bf (R)}. Then, define the control function $\tilde{\alpha}$ as follows:
\begin{equation*}
\tilde{\alpha}(t)=
\begin{cases}
\alpha_{\eps}(t), & \text{if}\ 0 \leq t \leq T_{\eps}
\\
a _{\eps }, & \text{if}\ t > T_{\eps}.
\end{cases}
\end{equation*}
Then, we have that
\begin{align*}
\lambda V_{\lambda} \leq\ & \lambda \int_{0}^{T_{\eps}}{e^{-\lambda s} L(X_{\alpha_{\eps}}^{x}(s), \alpha_{\eps}(s))\ ds} + \lambda \int_{T_{\eps}}^{\infty}{e^{-\lambda s}L(X_{\alpha_{\eps}}^{x}(T_{\eps}), a _{\eps })\ ds}
\\
\leq\ & \big(1-e^{-\lambda T_{\eps}}+e^{-\lambda T_{\eps}}\big)(I(x)+\eps).
\end{align*}
Therefore, as $\lambda \to 0^+$ by the arbitrariness of $\eps > 0$ we get the conclusion. \end{proof}

\begin{remark}
If  $L$ is independent to the second variable, i.e. $L(x,a)=L(x)$, then assumption   {\bf (R)} can be replaced by 
$$ \forall x \in  \overline\Omega, \;  \exists a \in A , \;  f(x,a) =0.$$
\end{remark}

\begin{remark}
Observe that if the the control system is controllable, i.e.
\begin{equation*}
\mathcal{R}(x)=\mathcal{R}(y), \quad x \in \OOO,\ y \in \OOO,
\end{equation*}
then the function $v^{*}$ is piece-wise constant. A typical example for which this holds is the case of a Tonelli Lagrangian.
\end{remark}

\begin{remark}
 It is well known that $\{\vall\}_{\ld>0}$ can be proved to be equi-Lipschitz by adding different assumptions to \eqref{eq:f} and \eqref{eq:L}.
 We recall here three main alternative conditions that can be used for this purpose, namely controllability, strict dissipativity, and nonexpansivity.
 \begin{enumerate}
\item \underline{Controllability}: suppose there exists $r>0$ such that
\begin{equation*}
B(0,r)\subset \overline{\rm co}\;f(x,A)\qquad\forall \in\bO.
\end{equation*}
Then even $\{\val\}_{\ld>0}$ is equi-Lipschitz (but only $\{\vall\}_{\ld>0}$ is equibounded), see, for instance, \cite[Proposition~II.2.3]{BCD}.
\item \underline{Strict dissipativity}: suppose there exists a constant $K>0$ such that
\begin{equation*}
\sup_{A \in A}
\inf_{B \in A} \big\langle f(x,a)-f(y,b),x-y\big\rangle\leqslant -K|x-y|^2\qquad\forall\,x,y\in\R^n\,.
\end{equation*}
Also in this case $\{\vall\}_{\ld>0}$ is equi-Lipschitz, see \cite[Theorem~VII.1.6]{BCD} and \cite{AG}.
\item \underline{Nonexpansivity}: suppose  for all $x,y\in\overline{\Omega}$ and  $a\in A$,  there exists
 $b\in A$ such that
\begin{equation*}
\begin{cases}
\hspace{.cm}
(i)&\big\langle f(x,a)-f(y,b),x-y\big\rangle\leqslant 0
\\
(ii)& L(x, a) - L(y, b)  \leqslant  K | x - y |
\hspace{.cm}
\end{cases}
\end{equation*}
for some constant  $K\geqslant 0$. Then it is shown in \cite{CQ} that the family $\{\vall\}_{\ld>0}$ is equi-Lipschitz on $\bO$ (see also \cite{QR}).
\end{enumerate}

\end{remark}

One might wonder whether, under the assumptions of the above corollary, one could find a function $u_0\in \Lip(\bO)$ such that $(u_0,v_0)\in \mathcal S_{\Lip}(\bO)$.
Our next example shows that this is not always the case.

\begin{example} \label{exemple}
Let $A = [0,1]\subset \R$. The interval  $ \overline{\Omega }=[0,1]$ is invariant for the state equation
\begin{equation}\label{eq-ex}
X'(t)=  -\al (t)  X (t) \qquad(t\geqslant 0).
\end{equation}
Now, consider the cost $$L(x,a) = 1- \sqrt{a}\,x\qquad (a,x\in [0,1]). $$
One can easily check that  $\val (0) = \frac{1}{\ld}$.  Moreover, we  have that
 \[  \val (x)  := \frac{1}{\ld} - \frac{x}{2 \sqrt{\ld}} \qquad (x\in [0,1])  \]
 is the viscosity solution of  equation \eqref{EDP},  which takes the  form
$$ - \ld \val (x) + \max_{a \in [0,1]} \big[\, a\, x\,\val' (x) + \sqrt{a}\,x-1\,\big]  =0 \qquad (x\in [0,1]). $$
Indeed, the above function is a classical solution at any interior point  and the corresponding inequalities can be verified at $x=0,1$ by straightforward computations.

Observe that, although $ \vall $ converges uniformly to  $v_0 \equiv 1 $ on $[0,1]$,  the family
$$
 \,\val(x)-\frac{v_0(x)}{ \lambda  } = - \frac{x}{2 \sqrt{\la}} \qquad (x\in [0,1]) $$
fails to be bounded below uniformly in $ \la$.
Therefore,  Proposition~\ref{speed1}~$(b)$ implies that there exists no $u_0\in \Lip(\bO)$ such that  $(u_0,v_0)\in \mathcal S_{\Lip}(\bO)$.  \qed
\end{example}

\subsection{Asymptotic analysis}
For any $x\in\bO\,,\;u\in\R\,,\;p,q\in\R^n $ let us define
the {\em joint Hamiltonian} $J$  as
 \begin{equation}\label{F0}
 J (x, u, p, q) = \max_{a \in A}\;\min\big\{\langle f(x,a), p \rangle\;,\;\langle f(x,a),q \rangle -L(x,a)+u    \big\}
\end{equation}
We say that a pair $(v,w ) $ of continuous functions on $\bO$ is a viscosity supersolution of
\begin{equation}\label{eq:J}
J(x, v, -\nabla v , -\nabla w) =0
\qquad\text{on}\;\;\overline { \Omega }
\end{equation}
if for all $x\in\bO$ we have that
\begin{equation*}
J(x, v(x, -p, -q) \geq 0
\qquad\forall p\in \partial^-_\Omega v(x)\,,\;\forall q\in \partial^-_\Omega w(x)\,.
\end{equation*}

\begin{proposition}\label{speed2} Let  $v,w\in \Lip(\bO)$ be such that $(v , w )$ is a viscosity supersolution  of \eqref{eq:J}.
Then the following holds true for any $\ld>0$:
\begin{itemize}
\item[$(a)$] for any $c\in\R$ the function
\begin{equation}
\label{eq:def_w^c}
w_c(x):=\,\frac{ v(x)}\ld\,+ w(x)+c\quad (x\in\bO)
\end{equation}
satisfies
\begin{equation}
\label{eq:w^c}
\ld w_c+H(x,-\nabla w_c)\geq \ld( w+c)  \quad\text{on}\;\;\overline { \Omega }
\end{equation}
in the viscosity sense;
\item[$(b)$] $\val$ satisfies the upper bound
\begin{equation}
\label{eq:upperb}
 \val(x)-\,\frac{ v(x)}\ld\leq  w(x)+\| w\|_\infty\qquad\forall\,x\in\bO.
\end{equation}
\end{itemize}
\end{proposition}

\begin{proof}
Proposition~\ref{mono}~$(i)$  below ensures that
for all $x \in  \overline{\Omega }$ there exists $ \alpha_x=\alpha \in {\cal A}$ such that
\begin{align*}
\begin{cases}
 w (x)  \geq  w  \big(\traj (t)\big) + \int_{0}^{t}  \big[L\big(\traj (s), \al (s)\big)-v\big(\traj (s)\big)\big]\ ds
 \vspace{.2cm}
 \\
 v (x)  \geq  v  (\traj (t) )
\end{cases}
\qquad  \forall t \geq 0.
\end{align*}
Hence,
\begin{equation*}
w_c (x)  \geq  w_c \big(\traj (t)\big) + \int _0 ^t  \big[L\big(\traj (s), \al (s)\big)-v\big(\traj (s)\big)\big]\ ds
\qquad  \forall t \geq 0
\end{equation*}
which, again by Proposition~\ref{mono}~$(i)$,  yields
$
 v+H(x,-\nabla w_c) \geq  0\; \text{on}\; \overline{\Omega}\,.
$
Thus, \eqref{eq:w^c} follows recalling \eqref{eq:def_w^c}.  Moreover, choosing $ c=\|u\|_\infty$ in \eqref{eq:def_w^c} we deduce from \eqref{eq:w^c} that $w_{ c}$ is a supersolution of \eqref{EDP}. Therefore, $w_{ c}\geq \val$ on $\bO$ by comparison and this yields \eqref{eq:upperb}.
\end{proof}

\section{Hamilton-Jacobi system}
In this section, we introduce a certain system of Hamilton-Jacobi  equations and investigate its properties and links with the limit value of the  discounted optimal control problem as the discount factor $\la$ goes to zero. From now on,  we make the following assumption which ensures the existence of optimal controls for problem \eqref{eq:val}:
\begin{equation}\label {eq:convex}
\big\{ (f(x,a), L(x,a) +r )~:~a \in A\,,\;  r>0 \big\}\; \mbox{is a convex subset of $  \rn \times \R $ for any $ x \in \rn$.}
\end{equation}

\subsection{Definition of the PDE system}
 
Let us consider  the system of Hamilton-Jacobi equations
\begin{equation*}
\begin{cases}
 h(x, -\nabla u ) =0
 \\
 u+H(x,-\nabla w) =0
 \\
 J(x, u, -\nabla u , -\nabla w) =0
\end{cases}
\eqno{(S)}
\end{equation*}
where $J$ is defined in \eqref{eq:J}. Again, $(S)$ is intended
to be satisfied in the viscosity sense by  a pair $(u , w ) $ of continuous functions on  $ \overline{\Omega}$.
Observe that, when the dynamics reduces to the differential equation $x'(t) = f(x(t))$, the third equation of $(S)$ follows from the first two.

\subsection{Properties of system $(S)$}
We start  by  proving the following uniqueness property.
\begin{proposition}\label{uniqueness}
Let $(u_1, w_1)$ and $(u_2,w_2) $ be solutions of $(S)$. Then $u _1 =u_2 $ on  $ \overline{\Omega}$.
\end{proposition}
\noindent Notice that no uniqueness is claimed for the second component of the solution to  $(S)$. Next, we provide an immediate consequence of Proposition \ref{mono} and, for this, we omit the proof here. 

\begin{lemma}\label{L3} Let $(u,w) $ be a solution of $(S)$. Then for all $ x \in  \overline{\Omega}$ and  $ \al \in {\cal A} $
\begin{equation}\label{l3}
\begin{cases}
 (i)
 &
 u(x) \leqslant  u ( \traj (t) )
 \\
 (ii)
 &\displaystyle
 w(x) \leqslant  \int _0 ^t \big[L\big( \traj (s), \al (s)\big) - u ( \traj (s))\big] ds + w ( \traj (t))
 ]
\end{cases}
\quad\forall  t \geq 0.
\end{equation}
\end{lemma}  
\begin{proof}[Proof of Proposition \ref{uniqueness}.]
Fix any $ x \in  \overline{\Omega}$ and take $ \al _x$ yielding equalities in \eqref{l3}  for $(u_1, w_1 )$.  Then, by \eqref{l3}-$(ii)$  applied to $ (u_2,w_2) $,  for all $ t \geq 0$ we have that
\begin{eqnarray*}
w_2 (x) -w_1 (x) \leqslant  \int _0 ^t \big[L\big( X^{x}_{\al _x} (s) , \al _x (s)\big) - u_2( X^{x}_{\al _x} (s)\big] ds + w_2  ( X^{x}_{\al _x}(t))\\
-\int _0 ^t\big[ L( X^{x}_{\al _x} (s), \al _x (s)) - u_1( X^{x}_{\al _x} (s)) \big]ds - w_1 ( X^{x}_{\al _x} (t)) \\
= \int _0 ^t  \big[u_1( X^{x}_{\al _x} (s)) - u_2( X^{x}_{\al _x} (s))\big] ds + w_2  ( X^{x}_{\al _x} (t)) - w_1 ( X^{x}_{\al _x} (t)).
\end{eqnarray*}
Since, by
\eqref{l3}-$(i)$,
\begin{equation*}
u_1( X^{x}_{\al _x} (s))  = u_1 (x)\quad\text{and}\quad u_2( X^{x}_{\al _x} (s))  \geq u_2 (x),
\end{equation*}
we conclude that
\begin{eqnarray*}
w_2 (x) -w_1 (x) \leqslant  t \big[u_1(x) - u_2 (x) \big] + w_2  ( X^{x}_{\al _x} (t)) - w_1 ( X^{x}_{\al _x} (t))\quad\forall  t \geq 0.
\end{eqnarray*}
Now, dividing by $t$ and letting $ t \to + \infty$ we obtain that $u_1 (x) \geqslant  u_2 (x)$ because $ w_1$ and $w_2$ are bounded.
The above argument is symmetric with respect to $u_1$ and $u_2$. So  $u_1 =u_2$.
\end{proof}

Solutions of $(S)$, whenever they exist, provide bounds for the speed of convergence of $\{\vl\}_{\lambda > 0}$. To do so, we introduce the function 
\begin{equation}\label{eq:wl}
\wl(x)=   \,\frac{\lambda \val(x)-v^{*}(x)}{ \lambda  } \qquad (x \in  \overline{\Omega})
  \end{equation}
where $v^{*}$ is given by \eqref{R1}.

\begin{theorem}[{\bf Rate of convergence}]\label{speed} Let  $(u , w ) $ be a solution   of $(S)$. Then for any $\lambda >0$
  \begin{equation}\label{eq:speed}
  w(x) - \| w  \| _ \infty \leqslant  \,\wl(x) \,  \leqslant   w (x) + \| w  \| _ \infty \quad\forall x \in  \overline{\Omega}.
  \end{equation}
\end{theorem}
\begin{proof}
Setting $$
w_{0}(x)=\frac{u(x)}{\lambda } + w(x) - \|w \| _ \infty \quad \big( x \in \OOO\big),  $$ we have that $w_{0}$ is a subsolution of \eqref{EDP}.  Since $ \val $ solves such an equation, by comparison we deduce that
$$\frac{u(x)}{\lambda } + w(x) - \|w\| _ \infty  \leqslant  \val(x), \quad \big( x \in \OOO\big).$$ This proves the lower bound in \eqref{eq:speed}.
On the other hand, setting
\begin{equation*}
	w^{0}(x)= \frac{u(x)}{\lambda } + w(x) + \|w  \| _ \infty, \quad \big(x \in \OOO\big),
\end{equation*}
by the definition of $J$ we deduce that $w^{0}$ is a supersolution of \eqref{EDP}. This yields the upper bound  in \eqref{eq:speed} and completes the proof.
\end{proof}

On the other hand, system $(S)$ may fail to have a solution even when $ \lambda \val $ converges uniformly as  $ \la \to 0 ^+ $.

\begin{example} We return to example \ref{exemple} where we have obtained

 \[   v^\star (x) =1 \mbox{ and } \val (x)  = \frac{1}{\ld} - \frac{x}{2 \sqrt{\ld}} \qquad (x\in [0,1]).  \]
Thus  the family
$$
\frac{\val(x)-v^ \star(x)}{\ld}  = - \frac{x}{2 \sqrt{\la}} \qquad (x\in [0,1]) $$
fails to be bounded uniformly in $ \la$. Hence 
  Proposition \ref{speed} implies that system ({\bf S}) has no solution. \qed
\end{example}

Now we give an easy example where we can deduce the convergence of $ \lambda V _ \lambda $ (which can be easily guessed) but moreover a speed of convergence of the form \eqref{eq:speed} which is harder to see at the first glance.

\begin{example} Let us consider an uncontrolled two dimensional linear system and a Lipschitz function $ L : \R^2 \to [0,1]$ which does not depend on controls, that is,
\[
\left(
\begin{array}{c}
 x' \\
  y'
\end{array}
\right)
=
\left(
\begin{array}{cc}
0  &   1  \\
-1  &   0
\end{array}
\right)
\left(
\begin{array}{c}
  x   \\
y
\end{array}
\right),
\]
for which $$\val (x,y ) :=   \int_{0}^ \infty e ^{- \lambda t}  L\big( X^x(t), Y^y(t)\big)dt.$$
Observe also that any ball $ \overline{\Omega} = B(0,R) $ is invariant under the above differential equation.
Now, define two functions $u,v:\R ^2\to\R$ in polar  coordinates  $(x,y) = (r \cos  \theta, r \sin \theta)$ as follows
\begin{eqnarray*}
u(x,y) = \frac{1}{2\pi} \int _ 0 ^{2 \pi} L(r \cos \sigma, r \sin \sigma ) d \sigma, \\
w(x,y) = \int _0 ^ \theta \big( - L(r \cos \sigma, r \sin \sigma ) + u (  r \cos \sigma, r \sin \sigma ) \big)\,d \sigma .
\end{eqnarray*}
A straightforward computation shows that $ (u, w) $ satisfies system $(S)$ (where, as previously observed in the uncontrolled case, the third equation is automatically satisfied).
Thus, in view of Proposition \ref{La} we have that $ \lambda \val $ converges uniformly to $u$ on $B(0,R$)  as $ \lambda \to 0 ^+ $. Moreover we have that
$$ | \lambda \val (x,y)
- u(x,y) | \leq 8 \pi \lambda , \, \, \forall (x, y ) \in B(0,R) ^2 , \, \forall \lambda >0 , $$ since clearly $ \| w \| _ \infty \leq 4 \pi $. \qed
\end{example}

In conclusion, we recall that it has  recently been proved in \cite{bib:CM} that, when controls act on acceleration, that is, for systems of the form
\begin{align*}
\begin{cases}
    \dot x(t) =& v(t)
     \\
    \dot v(t) =& u(t)
    \end{cases}
\end{align*}
with $u$ an $\R^{m}$-valued measurable function, even when the family $\{\lambda V_{\lambda}\}_{\lambda > 0}$ locally uniformly converges,  $\{ V_{\lambda}\}_{\lambda > 0}$ may not converge.

\begin{remark} 
 The above considerations concern the convergence of the Abel mean value   $\lim _ {\lambda \to 0 ^+ } \lambda V _ \lambda$. However, several other kinds of mean value may be of interest: in \cite{LQR},  it is proved that all possible cluster points are the same and a representation formula is given. However,  a Hamilton-Jacobi approach could be difficult  to develop for general means.
\end{remark}

\subsection{Asymptotic analysis}

In this section we assume that $\val$ is uniformly Lipschitz with respect to $\ld$ 
\begin{lemma} \label{wl1}
The function $\wl $  in \eqref{eq:wl} is continuous in $\overline {\Omega}$ and is a supersolution to
\begin{equation} \label{sursol-wl}
\ld\wl+ H ( x , -\nabla \wl) + v^{*}(x) = 0 \quad\text{on}\;\;\overline { \Omega }.
\end{equation}
\end{lemma}

\begin{proof}
Since $ \val $ is a viscosity supersolution to (\ref{EDP}), in view of Proposition \ref{mono} i)  (applied with $ \ell =L$) we know that for all $x \in  \overline{\Omega }$ there exists $\alpha \in {\cal A}  $ such that
for every $t \geq 0 $ we have $$
\val(x)  \geq e ^{- \lambda t } \val (\traj (t) ) + \int _0 ^t  e ^{- \lambda s}L  \big(\traj (s), \al (s)\big)ds.$$ Hence
\begin{eqnarray*}
\wl (x)  \geq  e ^{- \lambda t } \wl(\traj (t) ) +   \frac{1}{\lambda} \int _0 ^t  e ^{- \lambda s}L  \big(\traj (s), \al (s)\big) - v^{*} ( \traj (s)) ds \\ + \frac{1}{\lambda} [ -  v^{*}(x)  +   e ^{- \lambda t } v^{*} ( \traj (s)) +
\int _0 ^t  e ^{- \lambda s} v^{*} ( \traj (s)) ds
 ].
\end{eqnarray*} Observe that the  bracket last line  is nonnegative in view of Proposition \ref{mono} ii)  (applied with $ \ell =0$ and $ \lambda =0$), because  $v^{*}$ is a viscosity subsolution to \eqref{h}.  Consequently \[   \lambda \wl (x)  \geq  e ^{- \lambda t }  \lambda  \wl(\traj (t) ) +    \int _0 ^t  e ^{- \lambda s}L  \big(\traj (s), \al (s)\big) - v^{*} ( \traj (s)) ds,\] which means that $ \wl $ is a supersolution to  ( (\ref{sursol-wl}) again using   Proposition \ref{mono} i)  (applied with $ \ell =L -v^{*}  $).
\end{proof}

\begin{lemma} \label{vla}
The pair of continuous functions $(\vl,\wl) $  satisfies
\begin{equation} \label{sursol-vwl}
 J_\ld (x, \vl , \wl, -\nabla  \vl, - \nabla\wl ) \geqslant  \min \big\{ 0 ,  \ld\wl(x) \big\} =: \omega _ \la (x)
  \quad\text{on}\;\;\overline { \Omega }
  \end{equation}
  where we have set
   \begin{eqnarray}\label{Fla}
\lefteqn{J_\ld(x, u, w,  p, q)}
\\\nonumber
& =& \max_{a \in A}  \min
 \Big\{ \big\langle f(x,a ),p\big\rangle- \ld L (x,a)+ \ld u ~,~\big\langle f(x,a ),q\big\rangle- L (x,a)+ u+\ld w
 \Big\}
\end{eqnarray}
for all $ (x, u,w,p,q) \in  \rn \times \R \times \R \times   \rn \times   \rn$.
\end{lemma}
\noindent
Notice that $ \omega _ \la $ converges uniformly to zero as $ \la \to 0 ^+ $.

\smallskip
\begin{proof}
We have that
\begin{eqnarray*}
\lefteqn{J_\ld (x, \vl , \wl, -\nabla  \vl, - \nabla\wl )}
\\
&= &\max_{a \in A}  \min
 \Big\{ \big\langle f(x,a ),-\nabla\vl\big\rangle- \ld L (x,a)+\ld\vl ~,~\big\langle f(x,a ),-\nabla  \wl\big\rangle- L (x,a)+\vl+\ld \wl
 \Big\}
\\
&= &\max_{a \in A}  \min
 \Big\{\ld\big[ \big\langle f(x,a ),-\nabla\val\big\rangle-  L (x,a)+\ld\val \big]~,~\big\langle f(x,a ),-\nabla  \val\big\rangle- L (x,a)+\ld\val+\ld\wl
 \\
 &\quad&\hspace{10cm}+\,\frac 1\ld\, \langle f(x,a),\nabla v_0 \rangle \Big\}
 \\
&\geqslant&\max_{a \in A}  \min
 \Big\{\ld\big[ \big\langle f(x,a ),-\nabla\val\big\rangle-  L (x,a)+\ld\val \big]~,~\big\langle f(x,a ),-\nabla  \val\big\rangle- L (x,a)+\ld\val
 +\ld\wl
 \Big\}
\end{eqnarray*}
because $
 \langle f(x,a),\nabla v_0 \rangle \geqslant 0
$ for all $a \in A$
by Lemma~\ref{La}. Since, on account of \eqref{EDP},
\begin{equation*}
\max_{a \in A}\big\{ \big\langle f(x,a ),-\nabla\val\big\rangle-  L (x,a)+\ld \val\big\}=0\qquad(x \in \overline{\Omega})
\end{equation*}
we obtain \eqref{sursol-vwl}.
\end{proof}

\section{Examples} This section is devoted to control systems satisfying
\begin{equation} \label{HW}
\exists W \in C ( \overline{\Omega}) \mbox{ such that } H (x , \nabla W (x) ) \leqslant  0  \mbox{  in } \overline{\Omega}.
\end{equation}

Denote $ H_1 (x) := \max _{a \in A} L(x,a) $.

\begin{example}\label{ex1}
Let
  \begin{equation*} f(x,a) = f_0 (x) + \sum _{i=1} ^m a _i f_i(x)   \mbox{ with } A = \prod _{i=1} ^m [-1,1].\end{equation*}

Then
\begin{eqnarray}
H(x,p)=   \big\langle f_0(x),p \big\rangle   + \min _{a \in A} \{ \sum _{i=1} ^m a _i \big\langle f_i(x) ,p \big\rangle  + L(x,a) \}
\\ \leqslant   \big\langle f_0(x),p \big\rangle  -  \sum _{i=1} ^m a _i |\big\langle f_i(x) ,p \big\rangle|   +  \max _{a \in a} L(x,a).
\end{eqnarray}
Suppose now that there exists $i_0 \in \{ 1,2 , \ldots m\}$, $ \exists \delta >0 $, $ \exists g :  \rn \mapsto \R ^ d$ bounded and Lipschitz such that
\begin{equation}\label{E2} \left\{\begin{array}{l}
f_{i_0}(x)-g(x) = \nabla \phi (x) \\
 \big\langle f_0(x),p \big\rangle  + |\big\langle g(x) ,\nabla \phi (x) \big\rangle| + \delta \leqslant  \| \nabla \phi (x) \| ^2
\end{array} \right.
\end{equation}  for some $ \phi $ of class $C^1$.
Then $W(x) = k \phi (x) $ satisfies \eqref{HW} for $k$ large enough. Indeed
\begin{eqnarray*} H(x,p)  \leqslant
\big\langle f_0(x),p \big\rangle  -  \sum _{i=1} ^m a _i |\big\langle f_i(x) ,p \big\rangle|   + H_1(x) \\ \leqslant
k \{ \big\langle f_0(x),p \big\rangle  + |\big\langle g(x) ,\nabla \phi (x) \big\rangle|  -\| \nabla \phi (x) \| ^2 \} +  \| H_1 \| _ \infty  \\ \leq
- k \delta +   \| H_1 \| _ \infty
\end{eqnarray*} in view of \eqref{E2}. Therefore \eqref{HW} holds true for $k$ large enough. \qed
\end{example}

Now we discuss special cases of systems satisfying  \eqref{E2}

\begin{example} \label{ex2} Double Integrator.
In $ \R ^2 $ consider the system
\begin{equation*}
\begin{array} {lllll}
\left(\begin{array}{l}
x ' \\ y'
\end{array} \right)
& = &
\left(\begin{array}{l}
y \\ 0
\end{array} \right)
& + &
\al \left(\begin{array}{l}
0 \\ 1
\end{array} \right).
\end{array}
\end{equation*}
Here $A = [-1,1]$, $f_0 (x,y) = (y,0)$ and $f_1 (x,y) = (0,1)$. Then $f_1 (x,y) = \nabla \phi (x,y) $ where $ \phi (x,y) = y $ and $$
 \big\langle f_0, \nabla \phi  \big\rangle  - \| \nabla \phi \| ^2 = -1.$$
 So \eqref{E2} holds true  on $ \R ^2 $ with $ g =0 $ and  $ \delta =1$. \qed
\end{example}

\begin{example} \label{ex3} Harmonic Oscillator.
In $ \R ^2 $ consider the system
\begin{equation*}
\begin{array} {lllll}
\left(\begin{array}{l}
x ' \\ y'
\end{array} \right)
& = &
\left(\begin{array}{l}
y \\ -x
\end{array} \right)
& + &
\al \left(\begin{array}{l}
0 \\ 1
\end{array} \right).
\end{array}
\end{equation*}
Here $A = [-1,1]$, $f_0 (x,y) = (y,-x)$ and $f_1 (x,y) = (0,1)$.

Then $f_1 (x,y) = \nabla \phi (x,y) $ where $ \phi (x,y) = y $ and $$
 \big\langle f_0, \nabla \phi  \big\rangle  - \| \nabla \phi \| ^2 =-x  -1.$$
 So \eqref{E2} holds true  for $ x \geq \delta -1 $  ( $ \delta >0$ ) with $ g =0 $ and  $ \delta =1$. \qed

\end{example}

\begin{example} Nonholonomic integrator.
In $ \R ^3 $ consider the system
\begin{equation*}
\begin{array} {lllll}
\left(\begin{array}{l}
x _1' \\ x_2' \\ x_3 '
\end{array} \right)
& = & \al _1
\left(\begin{array}{l}
1\\ 0 \\ - x_2
\end{array} \right)
& + & \al _2
 \left(\begin{array}{l}
0 \\ 1 \\ - x_1
\end{array} \right).
\end{array}
\end{equation*}
Here $A = [-1,1]^2 $, $f_0 =0$,  $f_1 (x_1,x_2,x_3) = (1,0,x_2)$ and $f_2 (x_1,x_2,x_3) = (0,1, -x_1 )$.
Then  $f_1 = \nabla \phi +g $ with $\phi (x_1,x_2,x_3) = x_1$ and $g (x_1,x_2,x_3) = (0,0,x_2)$. Since $ \big\langle g, \nabla \phi  \big\rangle  =0 $, $ \| \nabla \phi \| =1 $, we have that   \eqref{E2} holds true with $ \delta =1$.
\qed
\end{example}

\section{Mean Payoff of Discrete Time Problems}
We now discuss the analogue of the present results, in the discrete time
setting. This applies in particular to Markov decision processes.
In this context, two classical tools have been developed
to characterize the mean payoff, in the special case of finite state
and action spaces. The first tool
is the notion of {\em invariant half-line} of dynamic programming operators,
which Kohlberg~\cite{kohlberg} showed to exist for polyhedral (two-player)
operators. Invariant half-lines also appeared, in the case of one player polyhedral
operators, in the monograph of Dynkin and Yushkevich~\cite{dynkin}. The other tool is the multichain
linear programming formulation of Denardo and Fox~\cite{denardofox}.
 We shall see that our earlier results extend these tools
to the Hamilton-Jacobi PDE setting, and that they also lead to extensions
of the invariant half-line approach to the discrete time case with
infinite state and action spaces.

It will be convenient to adopt an abstract setting.
Recall that an {\em AM-space with unit} is a Banach lattice $(X,\|\cdot\|,\leq)$
in which there is an element $e>0$, the unit, such that
$\|x\|= \inf\{t>0 \mid -te \leq x\leq te\}$. The
typical examples of AM-spaces with unit are $X=\cont(S)$,
the space of real valued continuous functions on a compact (Hausdorff) topological
space $S$, equipped with the sup-norm and with the standard partial
ordering of function, taking for the unit $e$
the constant function, or $L_\infty(\mu)$, the
space of equivalence classes of $\mu$-measurable functions
with respect to a positive measure $\mu$,
equipped with the essential sup norm.
In fact, by the Kakutani-Krein theorem, every Banach lattice is lattice
isomorphic and isometric to
a space $\cont(S)$, for a suitable choice of compact
topological space $S$, see
\cite[Theorem~8.29]{aliprantis}
or \cite[Chapter II, Theorem~7.4]{schaefer}.
Moreover, we shall shortly see that $S$ has a natural interpretation
as the state space of a game, in most applications.
So the reader
may like to think that $X=(\cont(S),\|\cdot\|_\infty,\leq)$ in the rest of this section,
although the explicit knowledge of the representation space
$S$ is not needed for several of our results.

We consider
an operator $T$ from $X$ to $X$ such that, for all $x,y\in X$ and $\lambda \in \R$,
\begin{align*}
x\leq y  &\implies T(x)\leq T(y)\\
T(\lambda e + x) &=  \lambda e + T(x)
\end{align*}
where $\leq$ denotes the standard partial order of $X$, and $e$ is the unit
vector of $X$. We shall say that $T$ is {\em order preserving}
when the former property holds, and that it {\em commutes with the addition of the unit} when the latter properties holds.
These two properties imply that $T$ is nonexpansive, meaning
that $\|T(x)-T(y)\| \leq \|x-y\|$. For the sake of brevity,
we make the following definition.
\begin{definition}
A map $T:X \to X$ is a
{\em Shapley operator} if it is order preserving
and commutes with the addition of the unit.
\end{definition}

The original example of Shapley operator arises when $X=\R^n$,
equipping every state $i\in S=[n]$ with finite sets of actions
$A_i$ and $B_i$,
payment maps $g_i: A_i \times B_i \to \R, (a,b)\mapsto g_i(a,b)$,
and a map $P_i: A_i \times B_i \mapsto \Delta([n])$, where
for any finite set $K$, $\Delta(K)$ denotes the set of probability
measures on $K$, which can be identified to nonnegative vectors $p=(p_k)_{k\in K}\in \R^K$ of sum $1$. Here, $(P_i(a,b))_j$ determines the probability
according to which the next state is chosen to be $j$, if the current state is
$i$. Then, the $i$th entry
of the corresponding Shapley operator is given by
\[
T_i(x) = \min_{\mu\in \Delta(A_i)}\max_{\nu\in \Delta(B_i)}\sum_{a\in A_i}\sum_{b\in B_i}
\big(g_i(a,b) + \sum_{j\in [n]}(P_i(a,b))_j x_j\big)\mu_a\nu_b \enspace .
\]
Such a $T$ is trivially order preserving and it commutes with the addition of the unit.
Conversely, order preserving maps that commute with the addition of the unit
are known to have somehow similar minimax representations, albeit with
possibly noncompact action spaces, see~\cite{Kol92,RS01b,1605.04518}.

The following discrete evolution equation
\[
v^k = T(v^{k-1}), \qquad k=1,2,\dots \qquad v^k\in X
\]
may be thought of as the analogue of the Hamilton-Jacobi evolution equation
\[
\partial_t V + H(x,-\nabla V)=0.
\]
We shall use the notation $T^k$ for the $k$th iterate of $T$, so that $v^k=T^k(v^0)$.

The analogue of the discounted value
function, defined as satisfying 
\[
\lambda V_\lambda + H(x,-\nabla V_\lambda) =0,
\]
for $\lambda>0$, is the vector $v_\discount\in X$, defined for a discount factor
$1-\discount$, with $0<\discount<1$  by
\[
T((1-\discount) v_\discount) = v_\discount
\]
The existence and uniqueness of $v_\discount$ follows from the nonexpansiveness
of $T$. We shall sometimes write $v_\discount=v_\discount(T)$ to emphasize
the dependence in $T$.
The correspondence between the discrete and continuous time setting
is gotten by setting $\exp(-\lambda)=1-\discount$.

The notion of invariant half-line was used by Kohlberg~\cite{kohlberg},
in a context in which $T:\R^n\to\R^n$ was only required to be nonexpansive.
Here, $T$ is also required to be order preserving, and so, it is convenient
to consider a one-sided version of the same notion, with an inequality
instead of an equality:
\begin{definition}
Suppose $T:X\to X$.
Then, a {\em sub-invariant} half-line is a map of the form $s\mapsto u+ s\eta$, $[0,\infty)\to X$, such that
\[
T(u+s\eta) \geq u+(s+1)\eta \enspace .
\]
We say that the vector $\eta$ {\em directs} the half-line.
The notion of {\em super-invariant} half-line is defined by reversing the latter inequality. A half-line is {\em invariant} if it is both sub- and super-invariant.
\end{definition}
It is useful to note that if $T$ is order preserving with a sub-invariant half--line $s\mapsto u+s\eta$,
\begin{align}
T^k(u) \geq u +k\eta, \qquad \forall k\geq 0
\label{e-pump}
\end{align}
whereas the equality in~\eqref{e-pump} holds unconditionally
on $T$ if $s\mapsto u+s\eta$ is an invariant half-line.
Observe that if $T$ is a Shapley operator,
it has always one sub-invariant half-line,
for instance, $s\mapsto -s\|T(0)\|$.

\begin{remark}
Special sub- or super-invariant half-lines arise when considering
{\em ergodic} sub- or super-eigenvectors $u$, with respect to a scalar
constant $\lambda$. The latter vectors satisfy $T(u) \geq u+ \lambda e$
or $T(u)\leq u+ \lambda e$. Then, the map $s\mapsto u +s \lambda e$ is readily
seen to be a sub- or super-invariant half-line.
\end{remark}

The following observation was made
by Kohlberg in a special case. It follows easily from the equality case in~\eqref{e-pump},
together with the nonexpansiveness of $T$. It shows that the existence
of an invariant half-line not only entails that the limit $\lim_k T^k(x)/k$
does exists and coincides with the offset $\eta$, but that $T^k(x)-k\eta$ stays
bounded as $k\to \infty$. A similar property holds true
for the vanishing discount limit.

\begin{proposition}[Compare with Corollary~2.2 of~{\cite{kohlberg}}]
  \label{prop-kohlberg}
Suppose that $T: X\to X$ is nonexpansive, and that it has
an invariant half-line. Then, for all $x\in X$, the sequence $(T^k(x) - k\eta)_{k\geq 0}$ remains bounded. In particular,
the limit $\lim_k T^k(x)/k$ does exist and coincides with the vector $\eta$.
Similarly, the sequence $v_\alpha-\alpha^{-1}\eta$ stays bounded as $\alpha\to 0^+$, and in particular $\lim_{\alpha\to 0^+} \alpha v_\alpha=\eta$.\hfill\qed
\end{proposition}

We next show that a ``one-sided'' version of Kohlberg's approach
holds for sub-invariant half-lines,
when $T$ is required in addition to be order preserving.
\begin{theorem}[Comparison principle]\label{th-1}
Suppose $T: X\to X$ is a Shapley operator,
and that $s\mapsto u+s\eta$ is a sub-invariant half-line of $T$.
Then, the map $\discount \mapsto v_\discount - \discount^{-1}\eta$
is bounded below as $\discount\to 0^+$, and the sequence
$(v^k - k^{-1}\eta)_{k\geq 0}$ is bounded below. In particular,
\begin{equation*}
\liminf_{\discount\to 0^+} \discount v_\discount \geq \eta \qquad \text{and } \qquad
\liminf_{k\to\infty} v^k/k \geq \eta \enspace .
\end{equation*}
\end{theorem}
Note in particular
that the notions of $\liminf$ and $\limsup$ are well defined
in Banach lattices~\cite{aliprantis}, and that
when $X=\cont(S)$, they agree with the the usual definitions
for functions.

To simplify the proof of \Cref{th-1}, it will be useful to consider, for all $u\in X$,
the conjugate map
\[
T_u (x):= -u + T( u + x) \enspace .
\]
The following lemma states some elementary properties of this map.
\begin{lemma}\label{lem-conj}
Let $T$ be a nonexpansive self-map of $X$, and let $u\in X$, then:
\begin{enumerate}
\item If $s\mapsto u+s\eta$ is a sub-invariant half-line of $T$, then
$s\mapsto s\eta$ is a sub-invariant half-line of $T_u$;
\item For all $x\in X$ and $k\geq 0$, $T_u^k(x) = -u+T^k(u+x)$;
\item $\|v_\discount(T)-v_\discount(T_u)\|\leq 2\|u\|$.
\end{enumerate}
\end{lemma}
\begin{proof}
We only show the last statement, the two first ones being trivial.  We have
\begin{align*}
\|v_\discount(T) - v_\discount(T_u)-u\|
& = \|T(\discount v_\discount(T))-T(u+\discount v_\discount(T_u))\|\\
& \leq \|\discount v_\discount(T) - u- \discount v_\discount(T_u)\|\\
& \leq \discount \| v_\discount(T) -u - v_\discount(T_u)\| + (1-\alpha)\|u\|
\end{align*}
and so, after simplifying,
\[
\|v_\discount(T) - v_\discount(T_u)+u\| \leq \|u\|
\]
from which the last statement follows.
\end{proof}

\begin{proof}[Proof of Theorem~\ref{th-1}]
We first show that $v_\discount -\discount^{-1}\eta$
is bounded below. Thanks to Lemma~\ref{lem-conj},
after replacing $T$ by $T_u$, we may assume that
the invariant half-line is of the form $s\mapsto s\eta$.

Let $z:= \discount^{-1}\eta$. Observe that $T((1-\discount) z)=
T((\discount^{-1}-1)\eta)\geq \discount^{-1}\eta =z$. Let $F(x)=T((1-\discount)x)$,
so that $z\leq F(z)$.
Since $T$ is order preserving, so is $F$, and so
$z\leq F(z) \leq F^2(z) \leq \dots\leq F^k(z)\to v_\discount$ as $k\to\infty$. We deduce that $v_\discount \geq \discount^{-1}\eta$, for all $\discount>0$, which
proves the part of the theorem concerning $v_\discount$.

The part concerning $v^k$ is simpler. Indeed, since $T$ is order
preserving, we deduce from the sub-invariant half-line property
that $T^k(u)\geq u + k \eta$ holds, so that the sequence $(T^k(u)-k\eta)_{k\geq 0}$ is bounded below, and since $T$ is non-expansive,
the sequence $(T^k(v^0)-k\eta)_{k\geq 0}$ must be bounded below
as well.
\end{proof}

The following is an immediate corollary of Theorem~\ref{th-1}.
\begin{corollary}\label{cor-1}
Suppose that $T$ is a Shapley operator $X \to X$,
and that it has both a sub-invariant and a super-invariant half-line with a common director vector $\eta$. Then, $v_\discount -\discount^{-1}\eta$ remains bounded as $\discount \to 0^+$. Moreover, the
same is true  of $v^k-k\eta$ as $k\to \infty$. In particular,
\[
\lim_{\discount\to 0^+} \discount v_\discount =
\lim_{k\to \infty} k^{-1} v^k = \eta \enspace .
\]
\end{corollary}

\section{Sub-invariant half-lines of Bellman operators}
The previous results apply to any Shapley operators
(two player problems).
We now consider the more special situation in which $T$ originates from a one player (minimization) problem. To do so, it is convenient to make
use of the representation space $S$.
For $i\in S$ and $x\in X$, we denote by
 $x_i$ the $i$ coordinate of $x$, so that $x$ is the continuous function
sending $i$ to $x_i$, and we denote by $T_i$ the $i$ coordinate
map of $T$, i.e., $T_i(x):=(T(x))_i$. We require each coordinate $T_i$ of $T$ to be a concave function from $X \to \R$.
We shall say that $T$ is {\em concave} from $X$ to $X$ when this is the case.
We make the following definition.
\begin{definition}\label{def-bellman}
An operator $T:X \to X$ is a {\em Bellman operator}
if it is both a Shapley operator and if it is concave.
\end{definition}
We shall see shortly that $T$ has a representation as a
dynamic programming operator of a one player game,
so that this definition coincides with the classical
one.

We denote by $\bounded(S)\supset X=\cont(S)$
the space of bounded functions from $S$ to $\R$.
This space can be equipped with the topology of uniform convergence,
and with the topology of pointwise convergence, which coincides with the order
topology~\cite[Lemma~8.17]{aliprantis}. To avoid any ambiguity, the symbol $\lim$ will refer to the uniform convergence topology, whereas the  symbol $\olim$ will refer to the order (or pointwise) topology. The convergence in the uniform convergence topology implies the convergence in the order topology, but not vice versa.
\begin{lemma}
Suppose that $T:X\to X$ is concave.
Then, the recession function of $T$, which is the map
$\hat{T}$ sending $y\in X$ to
\begin{align}
\hat{T}(y):=\olim_{s\to\infty} s^{-1}T(sy) \in \bounded(S)
\label{e-olim}
\end{align}
does exist, and it satisfies, for all $x\in X$,
\begin{align}
\hat{T}(y) = \inf_{s>0} s^{-1}(T(x+sy)-T(x)) \enspace .
\label{e-inf}
\end{align}
\end{lemma}
\begin{proof}
Fix now $y\in X= \cont(S)$. By a standard
result of convex analysis, for each $i\in S$, the recession
function $\hat{T_i}$ of the $i$ coordinate map of $T$, defined
by
\[\lim_{s\to\infty} s^{-1}T_i(sy)
\]
does exist and coincides with
\[
\inf_{s>0} s^{-1}(T_i(x+sy)-T_i(x))
\]
for any $x\in X$, see~\cite[Theorem~8.5]{Rock}. This implies that~\eqref{e-olim}
and~\eqref{e-inf} hold.
\end{proof}

\begin{proposition}[Characterization of sub-invariant half-lines]\label{prop-cns}
Assume that $T$ is concave from $X\to X$. Then, $s\mapsto u+s\eta$
is a sub-invariant half-line of $T$ if and only if
\[
T(u)\geq u+\eta ,\qquad \hat{T}(\eta)\geq \eta \enspace.
\label{e-charac-sub}
\]
\end{proposition}
\begin{proof}
The first condition is obtained by setting $s:=0$ in
$T(u+s\eta)\geq u+(s+1)\eta$, whereas the second one
is obtained from $\hat{T}(\eta)=\lim_{s\to\infty}s^{-1}T(u+s\eta)
\geq \eta$.
Conversely, by~\eqref{e-inf}, $T(u+s\eta) \geq T(u) + s\hat{T}(\eta)
\geq u+(s+1)\eta$.
\end{proof}

We now assume that the $i$th coordinate of $T$ is given
in the following explicit way
\begin{align}\label{e-represent}
T_i(x) = \inf_{a\in A_i} \big(r^a_i  + P^a_i x)
\end{align}
where $A_i$ is a non-empty set, $r^a_i \in \R$, and $P^a_i$ a positive
continuous linear map from $X\to \R$, such $P^a_i e =1$.
Recall that the dual space of $\cont(S)$ can be identified
to the space $\operatorname{ca}_r(S)$ of regular signed Borel measures with bounded variation on $S$. It follows
that $P^a_i x$ coincides with the integral of the continuous function $x$ with respect to a regular probability measure $\mu_i^a$ on the state space $S$.
Then, $T$ is the one
day dynamic programming operator of a Markov decision process
with state space $S$, set of actions $A_i$ in state $i\in S$.
The real number $r^a_i$ represents the instantaneous
cost in state $i$ for the action $a$, and $\mu^a_i$ is the probability
law according to which the next state is chosen.

When $T$ is of the form~\eqref{e-represent}, the conditions~\eqref{e-charac-sub}
characterizing sub-invariant half-lines are equivalent to an infinite system
of linear inequalities (indexed by $a\in A_i$).

\begin{remark}
An application of Legendre-Fenchel duality shows
that a representation of the form~\eqref{e-represent}
does exist as soon as $T$ is order preserving,
concave, and commutes with the addition of the unit,
which justifies the terminology ``Bellman operator''
in Definition~\ref{def-bellman}.
This was already noted in~\cite[Prop.~2.1]{spectral} in the finite
dimensional case.
We can write, applying Legendre-Fenchel duality to the convex
continuous map $-T_i : X\to \R$, and using the fact that the dual
of $X=\cont(S)$ is the space $\operatorname{ca}_r(S)$,
\begin{align}
-T_i(x) = \sup_{\mu\in\operatorname{ca}_r(S)} \big(\< \mu ,x> - (-T_i)^*(\mu)\big)
\label{e-pm}
\end{align}
where, for any convex function $f: X\to \R$, we denote by $f^*(\mu)= \sup_{x\in X}\big(\<\mu,x>-f(x)\big)$, $\operatorname{ca}_r(S) \to \R \cup\{+\infty\}$,  the Legendre-Fenchel transform of $f$.
Let $\mathcal{P}_i:= \{\mu \mid (-T_i)^*(-\mu)<\infty\}$
Since $T$ commutes with the addition of the unit,
$(-T_i)^*(-\mu) \geq \sup_{s\in \R}(\<-\mu, s e> +T_i(s e))
= \sup_{s\in \R}(s(1-\<\mu,  e>) + T_i(0))$,
and so $\mu\in \mathcal{P}_i$ implies $\<\mu,e>=1$.
Since $T$ is order preserving, we have $(-T_i)^*(-\mu)
\geq \sup_{x\geq 0} (\<-\mu,x> +T_i(x))
\geq \sup_{x\geq 0} \<-\mu,x> +T_i(0)$, so that
$(-T_i^*)(-\mu)<\infty$ implies that the measure
$\mu$ is positive. Hence~\eqref{e-pm}  can be rewritten as
\[
T_i(x) = \inf_{\mu\in\mathcal{P}_i} \big(\< \mu ,x> + (-T_i)^*(-\mu)\big)
\]
which is of the form~\eqref{e-represent}.
This corresponds to a Markov decision problem
in which in state $i\in[S]$, the action
of the player consists in choosing the probability
measure $\mu \in \mathcal{P}_i$ according to which the next state
is drawn, with an associated cost $(-T_i)^*(-\mu)$.

\end{remark}

It will be convenient, for any $a\in A_i$, to denote by $T_i^a$ the affine map
\[
T_i^a(x):= r_i^a +P_i^a x
\]
with recession function
\[
\hat{T}_i^a (x) = P_i^a x
\]
The following is the analogue of the complementary condition
for the Hamilton-Jacobi PDE, i.e., the third equation in System (S)
above:%
\begin{align}
\inf_{a\in A_i } \max(\hat{T}_i^a(\eta)-\hat{\eta}_i, {T}_i^a(u)-u_i-\eta_i) \leq 0 ,\qquad \forall i\in [n]
\label{e-compldisc}
\end{align}

The following provides an analogue of~\Cref{speed}, showing
a rate of convergence.%
\begin{theorem}
Suppose that $T$ is a Bellman operator.
Suppose in addition that $T(u)\geq u+\eta$ and that $\hat{T}(\eta) \geq \eta$
for some $u,\eta\in X$
and that the complementarity condition~\eqref{e-compldisc} holds,
the infimum being attained for each $i\in[n]$.
Then, $s\mapsto u+s\eta$ is an invariant half-line of $T$.
Moreover, the sequence $v_\alpha-\alpha^{-1}\eta$ stays bounded as $\alpha\to 0^+$.
\end{theorem}
\begin{proof}
It follows from Proposition~\ref{prop-cns} that $T(u+s\eta)\geq u+(s+1)\eta$.

Assume now that the infimum in~\eqref{e-compldisc} is attained
at some point $a\in A_i$ and that it takes a nonpositive value. Then,
we have $T_i^a(\eta)\leq \eta_i$ and $\hat{T}_i^a(u)\leq u_i+\eta_i$.
Since the map $T_i^a$ is affine, with linear part $\hat{T_i^a}$, we deduce that $T_i(u+s\eta) \leq T_i^a (u+s\eta)
= T_i^a(u) + sT_i^a(\eta) \leq u_i+ \eta_i + s\eta_i \leq u_i+(s+1)\eta_i$,
showing that $s\mapsto u+s\eta$ is an invariant half-line of $T$.

The final statement follows from~\Cref{prop-kohlberg}.
\end{proof}

\begin{theorem}\label{th-2}
Suppose that $T$ is a Bellman operator. Then, any accumulation point in the uniform
convergence topology of $\discount v_\discount$, as $\discount\to 0^+$,
coincides with
the supremum of the director vectors of sub-invariant half-lines
of $T$.
Moreover, the same is true of any accumulation point of the sequence
$(v^k/k)_{k\geq 0}$ with respect to the same topology.
\end{theorem}
We shall need the following standard observation.
\begin{lemma}\label{lem-acc}
Any accumulation point $\eta$ of $\discount v_\discount$ as $\discount\to 0^+$,
in the topology of uniform convergence, satisfies
\[\hat{T}(\eta)=\eta\enspace .
\]
The same is true of any accumulation point of the sequence $v^k/k$,
in the topology of uniform convergence.
\end{lemma}
\begin{proof}
Suppose that $\eta =\lim_k \discount_k v_{\discount_k}$ for some sequence
$\discount_k\to 0^+$.
Then, using the nonexpansiveness of $T$,
$\eta = \discount_k v_{\discount_k} +o(1) =
\discount_k T((\discount_k^{-1}-1)\discount_k v_{\discount_k})
+o(1)
=
\discount_k T(\discount_k^{-1}\discount_k v_{\discount_k} +O(1) )
+o(1)
=\discount_k T(\discount_k^{-1}\discount_k v_{\discount_k})
+ o(1)
=\discount_k T(\discount_k^{-1} \eta) + o(1)$,
and since the uniform convergence implies the convergence
in the order topology we get
$\eta=\lim_k \discount_k v_{\discount_k}
= \olim_k \discount_k v_{\discount_k}
= \olim_k  \discount_k T(\discount_k^{-1} \eta)
= \hat{T}(\eta)$.

Suppose now that $\eta = \lim_k T^{n_k}(0)/n_k$ for some sequence
$n_k$. We deduce from $\|T^{n_k +1}(0)-T^{n_k}(0)\|\leq \|T(0)\|$
that $\eta=\lim_{k\to \infty}n_k^{-1}T^{n_k + 1}(0)$.
Moreover,
\[ n_k^{-1}T^{n_k+1}(0)=n_k^{-1}T(n_k n_k^{-1}T^{n_k}(0))
= n_k^{-1}T(n_k \eta) + o(1)\enspace,
\]
since every map $n_k^{-1}T(n_k\cdot)$ is Lipschitz
of constant $1$. Letting $k$ tend to $\infty$,
we get $\eta = \hat{T}(\eta)$.
\end{proof}

\begin{proof}[Proof of Theorem~\ref{th-2}]
Let $\bar{\eta}$ denote the supremum of the director vectors of the invariant half-lines of $T$.  By Theorem~\ref{th-1},
\begin{align}
\liminf_{\discount\to 0^+} \discount v_\discount \geq \bar{\eta} \qquad\text{ and } \qquad
\liminf_{k\to\infty} v^k/k \geq \bar{\eta} \enspace .
\label{e-bound}
\end{align}
Let $\eta$ be the uniform limit of $\discount_k v_{\discount_k}$ for
some sequence $\discount_k\to 0$. Since the uniform convergence implies
the order convergence, it follows from the first inequality in~\eqref{e-bound}
that $\eta \geq \bar{\eta}$.
Moreover,
\begin{align*}
v_{\discount_k}& =T((1-\discount_k) v_{\discount_k})
= T(v_{\discount_k} - \discount_k v_{\discount_k} )
\enspace.
\end{align*}
Let us fix $\epsilon>0$. Since $\discount_k v_{\discount_k}$ converges, we can choose $k$ such that $\discount_k v_{\discount_k} \geq \eta - \epsilon e$.
Setting $u:= v_{\discount_k} - \discount_k v_{\discount_k}$, we get
$T(u)=u + \discount_k v_{\discount_k} \geq u+\eta -\epsilon e$. Morever,
by Lemma~\ref{lem-acc},
we have $\hat{T}(\eta)=\eta$, and since $T$ commutes
with the addition of the unit, $\hat{T}(\eta-\epsilon e)= \eta -\epsilon e$.
Then, it follows from Proposition~\ref{prop-cns} that $s\mapsto u + s(\eta -\epsilon e)$ is a sub-invariant half-line of $T$. We deduce that $\bar \eta \geq
\eta -\epsilon e$, and since this holds for all $\epsilon>0$,
$\bar \eta =\eta$.

Suppose  now $\eta$ is a uniform limit of a sequence $v^{n_k}/n_k$.
It follows from the second inequality in~\eqref{e-bound}
that $\eta \geq \bar{\eta}$.

Finally, let $\eta$ be any accumulation point of the sequence $v^k/k$.
We assume, without loss of generality, that $v^0=0$.
In particular, $\|v^p/p -\eta\|\leq \epsilon$, for some $p$,
hence, $T^p(0)\geq p(\eta-\epsilon e)$. We observed in Lemma~\ref{lem-acc}
that $\hat{T}(\eta)=\eta$.
Let $\tilde{\eta}:=\eta-\epsilon e$, so that
\begin{align}
T^p(0)\geq p\tilde{\eta}\enspace .\label{e-lb}
\end{align}
Since $\tilde{T}$ commutes with the addition of the unit, we
still have $\hat{T}(\tilde{\eta})=\tilde{\eta}$. Hence,
for all $x\in X$, and $t>0$, we get from~\eqref{e-inf},
\begin{align}
T(x+t\tilde{\eta})\geq T(x) +t\hat{T}(\tilde{\eta}) = T(x) + t \tilde{\eta} \enspace .\label{e-last}
\end{align}

Consider now
\[
u := \sup\big((p-1)\tilde{\eta},T(0)+(p-2)\tilde{\eta},\dots, T^{p-2}(0)+\tilde{\eta}, T^{p-1}(0)\big)
\]
For all $s>0$, we get, using successively the fact
that $T$ is order preserving and that~\eqref{e-last}
and~\eqref{e-lb} hold,
\begin{align*}
&T(u+s \tilde{\eta})  =
T\big(\sup\big((s+p-1)\tilde{\eta},T(0)+(s+p-2)\tilde{\eta},\dots, T^{p-2}(0)+(s+1)\tilde{\eta}, T^{p-1}(0)+s\tilde{\eta}\big)\big)\\
&\geq \sup\big(T((s+p-1)\tilde{\eta}),T(T(0)+(s+p-2)\tilde{\eta}),\dots, T(T^{p-2}(0)+(s+1)\tilde{\eta}), T(T^{p-1}(0)+s\tilde{\eta})\big)\\
& \geq
\sup(T(0)+(s+p-1)\tilde{\eta},T^2(0)+(s+p-2)\tilde{\eta},\dots, T^{p-1}(0)+(s+1)\tilde{\eta}, T^{p}(0)+s\tilde{\eta}))\\
&\geq
\sup(T(0)+(s+p-1)\tilde{\eta},T^2(0)+(s+p-2)\tilde{\eta},\dots, T^{p-1}(0)+(s+1)\tilde{\eta}, p\tilde{\eta}+s\tilde{\eta}))\\
&= u+(s+1) \tilde{\eta}
\end{align*}
since the terms in the last supremum, up to a circular permutation
and the subtraction by $(s+1)\eta$, are precisely those appearing
in the definition of $u$. We conclude that $\bar{\eta} \geq \tilde{\eta}=\eta-\epsilon e$,
and since this holds for all $\epsilon>0$,we get $\bar{\eta}\geq \eta$.
The other inequality follows from the second inequality in~\eqref{e-bound}.

\end{proof}

\begin{remark}
	Here we give some explanation making the last argument more intuitive:
if $T$ is order preserving, and if $T^p(v)\geq v$, then, one has $T(u)\geq u$, where $$u:=\sup(v,T(v),\dots,T^{p-1}(v)).$$ The vector $u$ in the above proof is constructed by adapting this idea, with a suitable shift by $\tilde{\eta}$.
\end{remark}
\begin{corollary}
Suppose that $T: \R^n\to \R^n$ is a Bellman operator. Then, the two limits
$\lim_{\discount\to 0^+}\discount v_\discount$ and $\lim_{k\to\infty} k^{-1}v^k$
coincide with the supremum of the director vectors of sub-invariant half-lines
of $T$.
\hfill\qed
\end{corollary}
\begin{remark}
A natural question is to give general conditions for the existence of an invariant half-line. So far, there seems to be two relatively well understood special
cases.
(i) Kohlberg showed in~\cite{kohlberg} that if
$T:\R^n\to\R^n$
is polyhedral and nonexpansive, then $T$ admits an invariant half-line.
(ii) Another special situation is the existence of an {\em ergodic eigenvector},
i.e., a vector $u\in X$ such that $T(x) =\lambda e + x$,
for some $x\in X$. Then, $s\mapsto u+s \lambda e$ is an invariant half-line.
The latter ergodic problem has been widely studied
in the setting of Perron-Frobenius theory and a number of sufficient
existence conditions are known, see in particular~\cite{Nus88,gaubertgunawardena,akiangauberthochart}.
\end{remark}
\begin{remark}
Kohlberg's result on the existence of invariant half-lines
for polyhedral nonexpansive maps is tight in
the following sense. Consider perturbing a polyhedral nonexpansive map
$T:\R^n\to \R^n$ in such a way
that $T$ is replaced by a non-polyhedral map $\tilde{T}$, also nonexpansive
(in the same norm), such that $T(x)-\tilde{T}(x) \to 0$ when $\|x\|\to\infty$. Then, it can be easily checked
that the limit $\eta:= \lim_k k^{-1}v^k(T)=\lim_{\discount\to 0^+}\discount v_\discount(T)$ remains unchanged after replacing $T$ and $\tilde{T}$. However,
it may be the case that $\tilde{T}$ has no invariant half-space
although $T$ does.
An example in which this situation arises is the following:
Let $T:\R^2\to \R^2$ be such that $T_1(x_1,x_2) =\max(x_1,x_2)$
and $T_2(x_2) = x_2$. Observe that $T(0)=0$, so that $T$ has
trivially an invariant half-line, with $\eta=0$.
Consider now $\tilde{T}_1(x)=
\log(e^{x_1}+ e^{x_2})$ and $\tilde{T}_2(x)=x_2$,
observe that $\tilde{T}$ is still sup-norm nonexpansive,
and that $\tilde{T}^k(0,0)= (\log(k+1), 0)$, so that
again $\lim_k k^{-1}v^k(\tilde{T})=0$. However, $\tilde{T}$ does not have any invariant half-line
$s\mapsto u+s\eta$. Indeed, we would have $\eta=\lim_k k^{-1}v^k(\tilde{T})=0$, and so,
$u$ would be a fixed point of $\tilde{T}$. However, for all $u\in \R^2$, we
have $\tilde{T}_1(u)>u_1$, implying that $\tilde{T}$ has no fixed point.
\end{remark}
\begin{remark}
We showed that if $T: X\to X$ has an invariant half-line
of the form $s\mapsto u+s\eta$, then the sequence  $(T^k(0)-k\eta)_{k\geq 0}$ is bounded. We may ask whether the converse property holds when $X=\R^n$.
It is known
that this is the case when $\eta$ is a scalar multiple from the unit,
so that $\eta=\lambda e$ for some $\lambda\in\R$.
Indeed, this follows from~\cite[Lemma~3]{gaubertgunawardena}, showing
that if the sequence $(T^k(0)-k\lambda e)_{k\geq 0}$ is bounded,
then, there is a vector $u$ such that $T(u)=\lambda e +u$.
\end{remark}
\begin{remark}
Rosenberg \& Sorin~\cite{rosenbergsorin} and
Sorin~\cite{sorin} considered a different notion
of sub- or super-invariant vectors associated to a Shapley
operator $T$.
Their construction can be transposed as follows
to the present setting.
Define, for $\delta>0$,
the set $\mathcal{C}^-_\delta$ of ``$\delta$-approximate sub-harmonic
vectors'' $\eta$ in $X$ such that
\[
T((\alpha^{-1}-1)\eta) \geq  -\delta e + \eta \alpha^{-1}
\]
for $\alpha$ small enough, or equivalently,
\[
T(s \eta) \geq -\delta e + (s+1) \eta
\]
for $s$ large enough. They showed that $\liminf_k v^k/k \geq \eta- \delta e$
if $\eta\in \mathcal{C}^-_\delta$, and deduced that
$\liminf_k v^k/k \geq \eta$ if $\eta \in \cap_{\delta>0} \mathcal{C}^-_\delta$.
The main discrepancy with the present notion of sub-invariant half-line,
$T(u+s\eta)\geq u+(s+1)\eta$, is that the ``basis point'' $u$ is missing.
In this way, there is no natural ``limit equation'' allowing one to
characterize effectively the elements of $\cap_{\delta>0} \mathcal{C}^-_\delta$.
The degree of freedom brought by the basis point allows one to remedy,
at least partly to this issue, in such a way that the classical fixed
point approaches (solving the ergodic equation $T(u) =u + \lambda e$,
and Kohlberg's invariant half-lines for polyhedral maps) are now recovered
as special cases of our approach.
\end{remark}
\begin{remark}
The fact that the limits $\lim_{\discount \to 0^+} \discount v_\discount =
\lim_{k\to \infty} k^{-1} v_k$ exist and coincide has been proved
by several authors in the case of Markov decision processes
with finite state space, see in particular Renault~\cite{Renault} and
Vigeral~\cite[\S~5.5.1 and 5.5.2]{vigeral}. By comparison,
the novelty in Theorem~\ref{th-2} is the characterization
of the limit in terms of sub-invariant half-lines.
\end{remark}
\begin{remark}
Theorem~\ref{th-2} gives a characterization of the mean payoff as the solution
of an infinite linear programming problem. When the state and action spaces are finite, we recover the finite linear program initially introduced by Denardo and Fox~\cite{denardofox}.
\end{remark}
\begin{remark}
Consider the situation in which $T:\R^n\to \R^n$ is a polyhedral Shapley operator. The latter arise from games with perfect information and finite action spaces. We may assume, without loss of generality, that $T$ is of the form
\[
T_i(x) = \min_{a\in A_i} \max_{b\in B_i} (g_i^{ab}+\sum_{j\in[n]}P_{ij}^{ab}x_j)
\]
where $A_i,B_i$ are finite sets, $r_i^{ab}\in\R$, and
$P_{ij}^{ab}\geq 0$ with $\sum_{j\in[n]} P_{ij}^{ab}=1$,
for all $a\in A_i$ and $b\in B_i$. Then, if $s\mapsto u+ s\eta$ is an invariant
half-line, one can easily check that $u,\eta$ satisfy the following lexicographic system
\begin{align}
\eta_i & = \min_{a\in A_i} \max_{b\in B_i}\sum_{j\in[n]}  P_{ij}^{ab}\eta_j \label{e-lex1}\\
\eta_i + u_i & = \min_{a\in \bar{A}_i} \max_{b\in \bar{B}_i(a)}
(g_i^{ab}+\sum_{j\in[n]}P_{ij}^{ab}u_j) \label{e-lex2}
\end{align}
where $\bar{A}_i$ denote the set of $a\in A_i$ which achieve the minimum
in~\eqref{e-lex1}, and for all $a\in \bar{A}_i$,
$\bar{B}_i(a)$ denote the set of $b\in B_i$ which achieve the maximum
in~\eqref{e-lex2}. Conversely, every solution $u,\eta$ of~\eqref{e-lex1},
\eqref{e-lex2} entails that the shifted map, $s\mapsto u+(s+s_0)\eta$
is an invariant half-line, for some $s_0\geq 0$. Thus, invariant half-lines
are determined by a lexicographic fixed point problem. Similarly, sub-invariant
half-lines can be determined by inequalities. If $s\mapsto u+s\eta$
is sub-invariant, then,
inequality $\leq$ holds in~\eqref{e-lex1} for all $i\in[n]$,
and the inequality holds
in~\eqref{e-lex2} for those $i\in [n]$ such that the equality in~\eqref{e-lex1}
holds. Conversely, these conditions imply that the
shifted map, $s\mapsto u+(s+s_0)\eta$
is a sub-invariant half-line, for some $s_0\geq 0$.
The lexicographic system was originally studied in the one-player case~\cite{denardofox,dynkin}.
\end{remark}

\section{Appendix : HJB and monotonicity properties} \label{Appendix}
This part is devoted to monotonicity properties and to a comparison result for \eqref{EDP} under the assumption that $\overline{\Omega}$ is invariant which is supposed throughout the section.

For a function $ \theta : \R ^n \mapsto \R$, the epigraph is denoted by
$ \Epi  \theta  :=\{ (x, y) \in \R ^N \times \R , \, \theta (x) \leq y \} $  while $ \Hypo  \theta  :=\{ (x, y) \in \R ^N \times \R , \, \theta (x) \geq y \} $ denotes the hypograph.
For a set $A \subset \R ^n $, the notation $ A ^-$ stands for the negative polar cone $ \{x \in \R ^n , \, \<x,a> \leq 0 , \; \forall a \in A \}$. When $ a \in A$, $ T _A (a) $ stands for the (contingent) tangent cone\footnote{The vector $u$ belongs to $ T _A (a) $ if and only if there exists sequences $ h _k \to 0 ^+$ and $ u _k \to u $ with $a + h_k u_k \in A$ for any integer $k $. } to $A$ at $a$ (cf.\ e.g.~\cite{aufr90}).

\begin{theorem} \label{comp} Assume that ($H _ \Omega $), \eqref{eq:f} and \eqref{eq:L} hold true.

Let $ \theta _1$ and $\theta _2$ be two bounded uniformly continuous function from $\overline{\Omega } $  to $ \R$.

If $\theta _1$ is a viscosity subsolution to \eqref{EDP} on $\overline{\Omega }$  and $ \theta _2$ is a viscosity supersolution on $\overline{\Omega }$  to \eqref{EDP} then

\begin{equation} \label{comp-in}
\forall x \in \overline{\Omega } , \, \theta _1 (x) \leq \val(x) \leq \theta _2 (x)
\end{equation}
\end{theorem}
From the above theorem we deduce the following
\begin{corollary}
The value function $\val$ is the unique BUC viscosity solution on $\overline{\Omega }$  to \eqref{EDP}.
\end{corollary}
Theorem \ref{comp} is related from the following monotonicity properties of the viscosity sub and supersolutions to the following Hamilton Jacobi equation
\begin{equation} \label{EDPgen}
\left\{ \begin{array}{l}
 \lambda V (x) + H \big(x, -\nabla V (x) \big) = 0\qquad(x \in \overline{\Omega})\\  \mbox{where }
 H(x, p) :=\max_{a \in A} \big\{ \big\langle f(x,a),p \big\rangle  -\ell(x,a) \big\}\qquad\forall\, (x,p) \in  \rn \times  \rn.
\end{array} \right.
\end{equation} with $\lambda \geq 0$ possibly equal to zero. Of course \eqref{EDPgen}   reduces to \eqref{EDP} when $\ell \equiv L$.
We consider that there exists some $k >0$ satisfying the following condition
\begin{equation}\label {eq:ell}
\left\{ \begin{array}{l}
|\ell (x,a)-\ell (y,a)|\leqslant  k|x-y|  \qquad\forall x,y\in \rn\,,\;\forall a\in A \\
\{ f(x,a), \ell (x,a) +r ), \, a \in A, \, r \geq 0 \,\} \mbox{ is a convex set.}
\end{array} \right.
\end{equation}

\begin{proposition} \label{mono}  Assume that ($H _ \Omega $), \eqref{eq:f} and \eqref{eq:ell} hold true.
Consider $ \theta  : \overline{\Omega } \mapsto \R$ a BUC function. Then

i) $\theta$  is a viscosity supersolution on $\overline{\Omega }$  to \eqref{EDPgen} if and only if
\begin{equation} \label{7.1}
\forall x \in  \overline{\Omega }, \exists  \alpha \in {\cal A} , \,  \forall t \geq 0, \,
\theta (x)  \geq e ^{- \lambda t } \theta  (\traj (t) ) + \int _0 ^t  e ^{- \lambda s}\ell \big(\traj (s), \al (s)\big)ds
\end{equation}

ii) $\theta$  is a viscosity subsolution on $\overline{\Omega }$  to \eqref{EDPgen} if and only if
\begin{equation} \label{7.2}
\forall x \in  \overline{\Omega }, \forall  \alpha \in {\cal A} , \,  \forall t \geq 0, \,
\theta (x)  \leq e ^{- \lambda t } \theta  (\traj (t) ) + \int _0 ^t  e ^{- \lambda s}\ell\big(\traj (s), \al (s)\big)ds
\end{equation}
\end{proposition}

\begin{proof}[Proof of Theorem \ref{comp}.]
  Suppose that $ \theta  _1$ and $ \theta _2 $ are as in Theorem \ref{comp}.

By Proposition \ref{mono} with $ \ell \equiv L$, we know that $ \theta _2$ satisfies \eqref{7.1}. By passing to the limit as $ t \to + \infty $ on \eqref{7.1} we obtain
\begin{equation*}
\theta _2 (x) \geq \int _0 ^ \infty  e ^{- \lambda s}L\big(\traj (s), \al (s)\big)ds.
\end{equation*}
The right hand side of the above inequality being clearly minorized by $\val (x) $ we obtain that $\val(x) \leq \theta _2 (x) $.

By Proposition \ref{mono} with $ \ell \equiv L$, we know that $ \theta _1$ satisfies \eqref{7.2}. By passing to the limit as $ t \to + \infty $ on \eqref{7.1} we obtain
\begin{equation*} \forall \al \in  {\cal A} , \,
\theta _1 (x) \leq \int _0 ^ \infty  e ^{- \lambda s}L\big(\traj (s), \al (s)\big)ds.
\end{equation*}
Hence passing to the infimum over $\al \in  {\cal A} $ this yields
\begin{equation*}
\theta _1 (x) \leq  \inf _{ \al \in  {\cal A} } \int _0 ^ \infty  e ^{- \lambda s}L\big(\traj (s), \al (s)\big)ds = \val (x).
\end{equation*} The proof of the theorem is complete. \end{proof}

Proposition \ref{mono} is based on several lemmas we state and prove now.

The first lemma concerns the relation between monotonicity properties of a function $ \theta$ and the viability and invariance of its epigraph and hypograph. It has been obtained in a nonconstrained case in \cite{AF1, FP1, CPQ1}.

\begin{lemma} \label{a}
Consider $ \theta  : \overline{\Omega } \mapsto \R$ a BUC function. Then

i) $\theta $ satisfies \eqref{7.1} if and only if $\Epi (\theta)  \cap ( \overline{\Omega } \times \R ) $ is viable for the following differential inclusion
\begin{eqnarray}\label{DI}
\left(\begin{array}{l}  x'(t) \\ y'(t) \end{array} \right) & \in &
\left\{\begin{array}{ll}  \left(\begin{array}{l}  f(x(t),a) \\  \lambda y(t) -\ell (x(t),a) -r \end{array} \right) &
a \in A, \, r \in [0,1]
 \end{array} \right\}
\end{eqnarray}
ii) $\theta $ satisfies \eqref{7.2} if and only if $\Hypo (\theta)  \cap ( \overline{\Omega } \times \R ) $ is invariant for the differential inclusion \eqref{DI}.
\end{lemma}

\begin{proof}
Observe that thanks to \eqref{eq:ell} the right hand side of the differential inclusion has compact convex values, moreover \eqref{eq:f} and \eqref{eq:ell} imply that it is a Lipschitz set-valued map.
Let us prove the part i) of the Lemma.  Suppose that $\Epi (\theta)  \cap ( \overline{\Omega } \times \R ) $ is viable  for \eqref{DI}.  This means that  for any $ (x, y ) \in \Epi \theta $, there exists a solution $ (x(\cdot), y (\cdot) ) $ to \eqref{DI} such that $ (x(0), y(0)) =( x ,y  ) $ and $(x(t),y(t)) \in \Epi \theta $ for any $ t \geq 0$.   By the Filippov measurable selection theorem (cf e.g. Theorem 8.2.10 in \cite{aufr90}) , there exists $\al \in {\cal A}$ and a mesurable function $r: \R \mapsto [0,1] $ such that
\begin{equation*}
\theta (\traj (t) ) \leq y(t) = e^{ \lambda t } y -  e^{ \lambda t } \int _0 ^t  e^{ -\lambda s} [ \ell \big(\traj (s), \al (s)\big)+r(s) ]ds , \forall t \geq 0.
\end{equation*}
Dividing the above inequality by $ e^{ \lambda t } $  and taking $ y = \theta (x) $ we obtained
\begin{eqnarray*}
\theta (y)  \geq e^{- \lambda t }  \theta (\traj (t) )  + e^{- \lambda t }  \int _0 ^t  e^{ -\lambda s} [ \ell \big(\traj (s), \al (s)\big)+r(s) ]ds
\\ \geq e^{- \lambda t }  \theta (\traj (t) )  + e^{- \lambda t }  \int _0 ^t  e^{ -\lambda s}  \ell \big(\traj (s), \al (s)\big) ds
\end{eqnarray*}
which is
 the wished monotonicity property \eqref{7.1} .

Conversely  suppose that $ \theta $ satisfies \eqref{7.1}. Fix $ x \in \overline{\Omega} $ and consider $ \al \in {\cal A} $ such that \eqref{7.1} holds true.
Let us define for any $ r \geq 0 $, $x(t) := \traj (t)$ and $y(t):= e^{ \lambda t }  \theta (y)-  e^{ \lambda t } \int _0 ^t  e^{ -\lambda s}  \ell \big(\traj (s), \al (s)\big)ds$. One can easily observe that $(x(\cdot), y(\cdot))$ is a solution of the differential inclusion \eqref{DI}. Moreover inequality \eqref{7.1} is clearly equivalent to $ \theta (x(t)) \leq y (t) $ for any $ t \geq 0$. This means that $\Epi (\theta)  \cap ( \overline{\Omega } \times \R ) $ is viable for \eqref{DI} the wished conclusion.

The proof of ii) goes in the same line analyzing what means that  $\Hypo (\theta)  \cap ( \overline{\Omega } \times \R ) $ is invariant for  \eqref{DI}. So we omit this part. \end{proof}

The following Lemma concerns the links between sub and super differential and the tangent cone to the epigraphs and hypographs. Without constraints the result has been obtained in \cite{FR} (cf also \cite{FP2}).

\begin{lemma} \label{b}   Consider $ \theta  : \overline{\Omega } \mapsto \R$ a BUC function. Then for all $x \ in \overline{\Omega} $

\begin{equation*}
p \in \partial _\Omega ^+ \theta (x) \Longleftrightarrow (-p, 1) \in [ T _ {\Hypo (\theta)  \cap ( \overline{\Omega } \times \R ) }  (x, \theta (x)) ] ^-
\end{equation*}
and
\begin{equation*}
p \in \partial _\Omega ^-  \theta (x) \Longleftrightarrow (p, -1) \in [ T _ {\Epi  (\theta)  \cap ( \overline{\Omega } \times \R ) }  (x, \theta (x)) ] ^- .
\end{equation*}
\end{lemma}

\begin{proof} We only prove the first equivalence, the second one being similar.  Take $x \ in \overline{\Omega} $  and  $ p \in \partial _\Omega ^+ \theta (x) $. Fix $ (u,v) \in  T _ {\Hypo (\theta)  \cap ( \overline{\Omega } \times \R )}(x, \theta (x)) $.  Since $ \theta $ is continuous, we have clearly $ v \leq 0$.  We claim that
\begin{equation*} - \<p,u> +v \leq 0. \end{equation*} Observe that our claim holds true if $u=0$. From now on we suppose $u \neq 0$.
By the very definition of the fact that  $ (u,v) \in  T _ {\Hypo (\theta)  \cap ( \overline{\Omega } \times \R )}(x, \theta (x)) $ there exists sequences $ h_k \to 0^+$ , $(u_k,v_k) \to (u,v)$ such that for any $k \geq 1$,
\begin{equation*}
x + h_ku_k \in \overline{ \Omega } , \; \theta (x+ h_k u_k) \geq \theta (x) + h_k v_k .
\end{equation*}
Since $ p \in \partial _\Omega ^+ \theta (x) $ we have
\begin{eqnarray*}
0 \geq \limsup _{y \to x, y \in  \overline{\Omega }} \frac{\theta(y)-\theta(x) -\<p, y-x> }{|y -x|} \\
\geq \limsup _{k }\frac{\theta(  x+ h_k u_k )-\theta(x) -\<p, h_k u_k> }{h_k|u_k|} \\
\geq \limsup _{k } \frac{\theta( x)+ h_k v_k- \theta(x) -\<p, h_k u_k> }{h_k|u_k|} \\
= \frac{v - \<p,u>}{|u| }
\end{eqnarray*} which proves our claim.

Conversely take $(-p, 1) \in [ T _ {\Hypo (\theta)  \cap ( \overline{\Omega } \times \R ) }  (x, \theta (x)) ] ^-$.
Fix a sequence $y_k \in \overline{ \Omega }  $ and $c \in \R$ such that  $ y_k \to x$ and
\begin{equation*}
\lim _k \frac{\theta(y_k)-\theta(x) -\<p, y_k-x> }{|y_k -x|} =
\limsup _{y \to x, y \in  \overline{\Omega }} \frac{\theta(y)-\theta(x) -\<p, y-x> }{|y -x|} >c
\end{equation*}
Up to a subsequence $ u_k = \frac{y_k -x}{|y_k -x|} $ converges to some $u$. By abuse of  notations we similarly denote the subsequence.
Note $h_k := |y_k -x| $.

For any $k$ large enough
\begin{eqnarray*}
\frac{\theta(y_k)-\theta(x) -\<p, y_k-x> }{|y_k -x|} >c.
\end{eqnarray*}
Hence
\begin{eqnarray*}
\theta(x+ h_k u_k)-\theta(x) -h_k\<p, u_k > \; \geq c h_k.
\end{eqnarray*}
 Hence $(x+h_k u_ k, \theta (x) + h_k [\<p,u_k> +c ] ) \in \Hypo (\theta)  \cap ( \overline{\Omega } \times \R ) $  for any $ k \geq 1$. Thus $ (u, \<p,u> +c) \in T _ {\Hypo (\theta)  \cap ( \overline{\Omega } \times \R ) }  (x, \theta (x))$ and consequently
\begin{equation*}
\<-p,u> + \<p,u> + c \leq 0.
\end{equation*}  So $ c \leq 0$ . Since this is valid for every $c$ such that $
\limsup _{y \to x, y \in  \overline{\Omega }} \frac{\theta(y)-\theta(x) -\<p, y-x> }{|y -x|} >c$ we obtain that $ p \in \partial _\Omega ^+ \theta (x) $. The proof is complete.
\end{proof}

\begin{proof}[Proof of Proposition \ref{mono}.]
Fix $ \theta $ a BUC function on $ \overline{ \Omega } $.

\underline{Proof of part i)} Suppose that  $ \theta $ is a supersolution on $ \overline{ \Omega } $. We wish to obtain that $ \theta$ satisfies \eqref{7.1}.
In view of Lemma \ref{a}, it is enough to show that  $ \Epi \theta \cap (\overline{ \Omega } \times \R) $  is viable for \eqref{DI}. For doing this, due to Viability Theorem ( Th. 3.2.4 in \cite{AV}), we have to check that
\begin{eqnarray}
\forall x \in \overline{ \Omega }, \; \forall (p,q) \in [ T _ {\Epi  (\theta)  \cap ( \overline{\Omega } \times \R ) }  (x, \theta (x)) ] ^- , \, \exists a \in A, \, \exists r \in [0,1] ,\\ \label{viab}
\<f(x,a), p> + q ( \lambda \theta(x)  - \ell (x,a) - r ) \leq 0.
\end{eqnarray}
Fix $x \in \overline{ \Omega } $ and $ (p,q) \in [ T _ {\Epi  (\theta)  \cap ( \overline{\Omega } \times \R ) }  (x, \theta (x)) ] ^- $.
First observe that  $q \leq 0$.
Indeed
since
$ \{x\} \times [\theta (x) , + \infty ) \subset \Epi  (\theta)  \cap ( \overline{\Omega } \times \R ) $ then
$ \R ^n \times [0, + \infty ) = T _{ \{x\} \times [\theta (x) , + \infty ) } (x , \theta (x)) \subset T _ {\Epi  (\theta)  \cap ( \overline{\Omega } \times \R ) }  (x, \theta (x))$.  So $[T _ {\Epi  (\theta)  \cap ( \overline{\Omega } \times \R ) }  (x, \theta (x))]^- \subset [ \R ^n \times [0, + \infty )] ^- = \{0\} \times  (- \infty, 0]$.

Consider first the case where $q \neq 0$. Then $ (\frac{p}{|q|} ,-1) \in [ T _ {\Epi  (\theta)  \cap ( \overline{\Omega } \times \R ) }  (x, \theta (x)) ] ^- $ and by Lemma \ref{b} we obtain $ \frac{p}{|q|}  \in \partial ^- _ \Omega \theta (x)$.

Because $ \theta $ is a supersolution on  $\overline{\Omega } $, we obtain
\begin{equation*}
\lambda \theta (x) + \max _{a \in A} \{ - \<f(x,a), \frac{p}{|q|} > - \ell (x,a) \} \geq 0.
\end{equation*} Hence there exists some $a \in A$ such that
\begin{equation*}
\<f(x,a) ,p > + q (\lambda \theta (x) - \ell (x,a) ) \leq 0 ,
\end{equation*}  which is the wished relation \eqref{viab} (with $r=0$).

Consider now the case $q=0$. The equation \eqref{viab}  being obvious if $p=0$, we suppose $ p \neq 0$. Observe that $ (p,0) \in T_{\Epi \bar{ \theta} } (x, \bar{\theta} (x))$ where
\begin{eqnarray*}
\bar{\theta} (y) & =& \left\{ \begin{array}{lll} \theta(y) & \mbox{if} & y \in \overline{\Omega }  \\
\| \theta \| _ \infty +1 & \mbox{else} & \end{array} \right.
\end{eqnarray*} is a bounded lower semicontinuous function.

Using a version (Lemma 3.4 in \cite{PQ}) of a result due to Rockafellar \cite{Rock}, we know that there exists converging sequences $x_k \to x$, $(p_k, q_k) \in [ T _ {\Epi  (\bar{\theta})   }  (x_k, \bar{\theta }(x_k)) ] ^- $  such that
\begin{equation*}
q_k <0 , \, (p_k,q_k) \to (p,0) , \; \bar{\theta } (x_k) \to \bar{\theta }(x) .
\end{equation*}

Observe that if $y \notin \overline{\Omega }$ we have $ T _ {\Epi  (\bar{\theta})   }  (y, \bar{\theta }(y)) = \R ^n \times [0, + \infty)$. So
$[T _ {\Epi  (\bar{\theta})   }  (y, \bar{\theta }(y))]^- = \{0\} \times (-\infty , 0]$. Consequently since we have supposed $p \neq 0$ we deduce that for $k$ large enough $p_k \neq 0$, so $x_k \in \overline{\Omega }$  and $\bar {\theta} (x_k) = \theta (x_k)$. Hence

\begin{equation*}
T _ {\Epi  (\bar{\theta})   }  (x_k, \bar{\theta }(x_k)) ] ^- =T _ {\Epi  (\theta)  \cap (\overline{\Omega } \times \R  }  (x_k, \theta (x_k)) ] ^-
\end{equation*}
By Lemma \ref{b}, we deduce that $ \frac{p_k}{|q_k|} \in  \partial ^ - _ \Omega \theta (x_k)$. The function $\theta$ being a supersolution, we obtain
\begin{equation*}
\lambda \theta (x_k) + \max _{a \in A} \{ - \<f(x_k,a), \frac{p_k}{|q_k|} > - \ell (x_k,a) \} \geq 0.
\end{equation*} Multiplying by $q_k$ and letting $k$ tend to $ \infty$, this yields
\begin{equation*}
\max _{a \in A} \< -f(x,a),p > =0.
\end{equation*}
Hence we have obtained \eqref{viab} for $r=0$ and $q=0$.

Conversely assume that $\theta$ is a BUC function satisfying \eqref{7.1}.   Take  $ x \in \overline{\Omega }  $ and $ p \in  \partial _ - ^ \Omega \theta (x)$. So by Lemma \ref{b}
$(p, -1) \in [ T _ {\Epi  (\theta)  \cap ( \overline{\Omega } \times \R ) }  (x, \theta (x)) ] ^-$. In view of Lemma \ref{a} and of the viability theorem
( Th. 3.2.4 in \cite{AV}), there exists $a \in A$ and $r \in [0,1] $ such that \begin{equation*}
\<f(x,a), p> + (-1) ( \lambda \theta(x)  -\ell (x,a) - r ) \leq 0.\end{equation*}
Thus
\begin{equation*}
\lambda \theta (x) + H(x, -p) \geq r \geq 0,
\end{equation*} and consequently $\theta$ is a supersolution on $\overline{\Omega }$.

\underline{Proof of part ii)} This part being very similar to part i) we only sketch its proof.
 Suppose that  $ \theta $ is a subsolution on $ \overline{ \Omega } $. We wish to obtain that $ \theta$ satisfies \eqref{7.2}.
In view of Lemma \ref{a}, it is enough to show that  $ \Hypo \theta \cap (\overline{ \Omega } \times \R) $  is invariant for \eqref{DI}. For doing this, due to Invariance Theorem ( Th.2 in \cite{FP2} or cf \cite{AV}), we have to check that
\begin{eqnarray}
\forall x \in \overline{ \Omega }, \; \forall (p,q) \in [ T _ {\Hypo  (\theta)  \cap ( \overline{\Omega } \times \R ) }  (x, \theta (x)) ] ^- , \, \forall a \in A, \, \forall r \in [0,1] ,\\ \label{inv}
\<f(x,a), p> + q ( \lambda \theta(x)  - \ell (x,a) - r ) \leq 0.
\end{eqnarray}
Fix $x \in \overline{ \Omega } $ and $ (p,q) \in [ T _ {\Hypo  (\theta)  \cap ( \overline{\Omega } \times \R ) }  (x, \theta (x)) ] ^- $.
First observe that  $q \geq 0$. Then as in part i) we prove directly than $ -p / q \in \partial ^+ _ \Omega (x) $ when $q \neq 0$ and  we deduce \eqref{inv}. For $q =0$ we use the result of Rockafellar and the end of the proof proceeds exactly as in part i).
\end{proof}

\begin{remarks}
\begin{enumerate}
\item As already noticed, the classical notion of viscosity with constraints requires that the function $ \theta$ is a supersolution on $ \overline{ \Omega } $ and a subsolution in $ \Omega$. One can deduce from the above proofs that if $\theta$ is a subsolution on $ \overline{ \Omega } $  then necessarily $  \overline{ \Omega } $ is invariant by the control system \eqref{eq:traj}. So in the context of assumption ($H _ \Omega$) we have demonstrate that the notion of viscosity solution of Definition \ref{visc} are appropriate for comparison and uniqueness result without need of interior cone condition classically assumed \cite{BCD}.

\item In Theorem \ref{comp}, we do not need the invariance assumption ($H _ \Omega$) for proving $ \val (x) \leq \theta _2 (x) $, only any assumption ensuring  that $ \val $ is BUC is enough \cite{Soner1}.
\end{enumerate}

\end{remarks}

{\bf Acknowledgments:} This research was funded, in part, by l’Agence Nationale de la Recherche (ANR), project ANR-22-CE40-0010. P. Cannarsa and C. Mendico were partially supported by Istituto Nazionale di Alta Matematica, INdAM GNAMPA project 2022 and INdAM-GNAMPA project 2023, and by the MIUR Excellence Project awarded to the Department of Mathematics, University of Rome Tor Vergata, CUP E83C23000330006.

\bibliographystyle{abbrv}

\end{document}